\documentclass[11pt]{amsart}
\usepackage{amsmath, amssymb}

\newcommand{\transv}{\mathrel{\text{\tpitchfork}}}
\makeatletter
\newcommand{\tpitchfork}{%
  \vbox{
    \baselineskip\z@skip
    \lineskip-.52ex
    \lineskiplimit\maxdimen
    \m@th
    \ialign{##\crcr\hidewidth\smash{$-$}\hidewidth\crcr$\pitchfork$\crcr}
  }%
}
\makeatother

\usepackage[demo]{graphicx}
\usepackage{subcaption} 

\usepackage{tikz} 
\usepackage{wrapfig}
\usetikzlibrary{arrows,shapes,positioning}
\usetikzlibrary{decorations.markings}
\tikzstyle arrowstyle=[scale=1]  
\tikzstyle directed=[postaction={decorate,decoration={markings,
		mark=at position .65 with {\arrow[arrowstyle]{stealth}}}}]
\tikzstyle reverse directed=[postaction={decorate,decoration={markings,
		mark=at position .65 with {\arrowreversed[arrowstyle]{stealth};}}}]
\usepackage{float} 
\usepackage[margin=1.0in]{geometry} 
\usepackage[skip=2pt,font=scriptsize]{caption}

\makeatletter
\newcommand*\bigcdot{\mathpalette\bigcdot@{.5}}
\newcommand*\bigcdot@[2]{\mathbin{\vcenter{\hbox{\scalebox{#2}{$\m@th#1\bullet$}}}}}
\makeatother

\usepackage{amsmath}
\usepackage{amsthm}
\usepackage{amsfonts}
\usepackage{xcolor}
\usepackage[english]{babel}

\newcommand{\cs}{$\clubsuit$} 
\newcommand{\red}{\color{red}}
\newcommand{\blue}{\color{blue}}

        \definecolor{pink}{rgb}{1,0,1}

\newcommand{\edit}[1]{ {\red \cs #1 \cs}}

\usepackage[latin9]{inputenc}
\usepackage{amsmath}
\usepackage{amsfonts}
\usepackage{amssymb}
\usepackage{amsthm}

\newtheorem{theo}{Theorem}[section]
\newtheorem{prop}[theo]{Proposition}
\newtheorem{coro}[theo]{Corollary}
\newtheorem{lemm}[theo]{Lemma}

\theoremstyle{definition}
\newtheorem{def1}[theo]{Definition}
\newtheorem{rema}[theo]{Remark}

\newcommand{\nwc}{\newcommand}
\nwc{\eps}{\epsilon}
\nwc{\vareps}{\varepsilon}
\nwc{\Oph}{\operatorname{Op}_\hbar}
\nwc{\la}{\langle}
\nwc{\ra}{\rangle}

\nwc{\mf}{\mathbf} 
\nwc{\blds}{\boldsymbol} 
\nwc{\ml}{\mathcal} 

\nwc{\defeq}{\stackrel{\rm{def}}{=}}

\nwc{\cE}{\ml{E}}
\nwc{\cN}{\ml{N}}
\nwc{\cO}{\ml{O}}
\nwc{\cP}{\ml{P}}
\nwc{\cU}{\ml{U}}
\nwc{\cV}{\ml{V}}
\nwc{\cW}{\ml{W}}
\nwc{\tU}{\widetilde{U}}
\nwc{\IN}{\mathbb{N}}
\nwc{\IR}{\mathbb{R}}
\nwc{\IZ}{\mathbb{Z}}
\nwc{\IC}{\mathbb{C}}
\nwc{\IT}{\mathbb{T}}
\nwc{\tP}{\widetilde{P}}
\nwc{\tPi}{\widetilde{\Pi}}
\nwc{\tV}{\widetilde{V}}
\nwc{\supp}{\operatorname{supp}}
\nwc{\rest}{\restriction}

\newcommand{\wt}{\widetilde}

\newcommand{\R}{{\mathbb R}}

\newcommand{\Z}{{\mathbb Z}}

\renewcommand{\d}{\partial}

\newcommand{\codim}{{\operatorname{codim\,}}}

\renewcommand{\phi}{\varphi}

\newcommand{\bcal}{\mathcal{B}}

\newcommand{\dcal}{\mathcal{D}}

\newcommand{\gcal}{\mathcal{G}}

\newcommand{\lcal}{\mathcal{L}}

\newcommand{\tcal}{\mathcal{T}}
\newcommand{\ucal}{\mathcal{U}}

\newcommand{\al}{\alpha}

\newcommand{\ep}{\varepsilon}

\newtheorem{lem}[theo]{{\sc Lemma}}
\newtheorem*{main-theorem}{Theorem}

\newcommand{\fakesection}[1]{%
	\par\refstepcounter{section}
	\sectionmark{#1}
	\addcontentsline{toc}{section}{\protect\numberline{\thesection}#1}
}

\title{Centrally symmetric analytic  plane domains are spectrally determined in this class}

\author[Hezari]{Hamid Hezari }
\address{Department of Mathematics, UC Irvine, Irvine, CA 92617, USA} \email{hezari@math.uci.edu}

\author[Zelditch]{Steve Zelditch}
\address{Department of Mathematics, Northwestern University, Evanston, IL 60208, USA} \email{zelditch@math.northwestern.edu}

\date{\today}

\begin{document}

\begin{abstract} We prove that, under some generic  non-degeneracy assumptions,  real analytic, centrally symmetric plane domains are determined by their Dirichlet (resp. Neumann)   spectra. We prove that the conditions are open-dense for real analytic convex domains. One   step is to use a Maslov index calculation to show that the second derivative of the defining function of a  centrally symmetric domain at the endpoints of a bouncing ball orbit is a  spectral invariant. This is also true for up/down symmetric
domains, removing an assumption from the  proof in that case.

\end{abstract}

\maketitle

\fakesection{Introduction}

A plane domain $\Omega \subset \R^2$ is called `centrally symmetric' if it is invariant under the isometric involution
 $\sigma (x, y) =  (-x, -y). $ In this article, it is shown that    bounded, simply-connected,  centrally symmetric real analytic  plane domains
 satisfying a finite number of additional conditions  
 are determined by their Dirichlet, resp. Neumann, Laplace spectra among this family of domains. 
 Every centrally symmetric bounded plane domain has at least one (two if it is star-shaped) $\sigma$-invariant `bouncing ball' orbit for the billiard flow, i.e. straight
 line segments hitting $\partial \Omega$ orthogonally at both endpoints, which contain the origin. The corresponding line segment is
 invariant under $\sigma$ (which reverses its orientation).  We fix a bouncing ball orbit and denote it by $\gamma$;
 the length of the corresponding line segment in $\Omega$ is denoted by $L$.  The main condition
 is that the lengths $2L, 4L$ of a bouncing ball orbit have   multiplicity one in the length spectrum. The result parallels that
 of \cite{Z09} for up-down symmetric plane domains, i.e. where the isometric involution is $(x, y) \to (x, -y)$.

 \begin{figure} [H]
 	\begin{minipage}{.5\textwidth}
 		\begin{tikzpicture}[scale=0.50]
 		\draw [black] plot [smooth cycle] coordinates {(4,0.5) (2.9,1.9) (1.5, 2.7) (0, 3) (-1.5,2.7)  (-2,1.5) (-4,-0.5) (-2.9,-1.9) (-1.5, -2.7) (0, -3) (1.5, -2.7) (2,-1.5)};
 		\draw[blue] (0,3)--(0,-3);
 			\node at (.3, 1.7) [align=center]{$\gamma$};
 				\end{tikzpicture}
 		\centering
 		\caption{ A centrally symmetric domain.}
 		\label{CS}
 	\end{minipage}%
 	\begin{minipage}{.5\textwidth}
 		\begin{tikzpicture}[scale=0.50]
 		\draw [black] plot [smooth cycle] coordinates {(3.7,0) (2.83,1.83) (1.5, 2.7) (0, 3) (-1.5,2.7)  (-2,1.5) (-4,0) (-2,-1.5) (-1.5, -2.7) (0, -3) (1.5, -2.7) (2.83,-1.83)};
 	    \draw[blue] (0, 3)--(0,-3);
 	    	\node at (.3, 1.7) [align=center]{$\gamma$};
 			\end{tikzpicture}
 		\centering
 		\caption{ An up-down symmetric domain.}
 		\label{AS}
 	\end{minipage}
 \end{figure}
 
  To
  state the result precisely, we need some notation. We denote by $P_{\gamma}$  the  linear Poincar\'e map of $\gamma$. An orbit $\gamma$ is non-degenerate if $\det (I - P_\gamma) \neq 0$.  When the orbit is elliptic, the eigenvalues of $P_\gamma$
  are of modulus one and of the form $\{e^{ i \alpha}, e^{-i\alpha}\}$, $0 < \alpha \leq \pi$, and when it is the hyperbolic its eigenvalues $\{e^{\alpha}, e^{-\alpha}\}$, $\alpha > 0$, are real and they are never roots of unity in the non-degenerate case. We may rotate 
 $\Omega$, keeping the origin as the center of the symmetry, to make $\gamma$ the vertical $y$-axis.
 Then locally near the vertices of $\gamma$, $\partial \Omega$ consists of two graphs, $$ \{y = f(x)\} \cup \{y = -f(-x)\}, $$ lying above, resp. below, the horizontal axis.

Modifying the definition of \cite[Section 1.1.1]{Z09}, we define 
  $\mathcal D_{L}$, to be the class of simply-connected centrally symmetric  real analytic plane domains $\Omega$ satisfying the conditions:
\begin{enumerate}\label{CONDITIONS}
\item  There is a  non-degenerate bouncing ball orbit $
\gamma$ of length $L_{\gamma} = 2 L$ through the origin.

\item In the the elliptic case,  the eigenvalues $\{e^{i \alpha}, e^{-i \alpha}\}$  of the linear Poincare map $P_{\gamma}$ satisfy that $ \cos
\frac{\alpha}{2} $ does not belong to the `bad set' ${\mathcal B} = \{0,  \frac12, 1 \}$. 

\item $f^{(3)}(0) \neq 0$. 

\item The lengths $2L, 4L$ of $\gamma, \gamma^2$, have  multiplicity one in the length spectrum Lsp$(\Omega)$ and $4L \neq |\d \Omega|$. 

\end{enumerate}

When the orbit is elliptic, its eigenvalues
$\{e^{\pm i \alpha}\}$ are of modulus one and we require $\cos \frac{\alpha}{2}$ to lie  outside a certain bad set $\bcal$. In the hyperbolic case, its eigenvalues are of the form $\{e^{\pm \alpha}\}$ and we require no condition other than the non-degeneracy assumption $\alpha \neq 0$.  The role of the bad set ${\mathcal B}$ will become clear during the proof;   the angles being `non-bad'  eliminates  angle parameters where  certain
functions appearing the wave trace invariants fail to be independent. 

 It is proved in \cite{PS87} that the set of domains satisfying the conditions is generic in the $C^{\infty}$ topology. It does not automatically follow   that the set of real analytic domains satisfying the conditions (and which additionally are assumed to be
centrally symmetric and have a bouncing ball orbit of a prescribed length $L$) are generic  (with the relative topology).  In Proposition \ref{open dense and generic}, we prove that in the
space of centrally symmetric {\it convex } real analytic domains, the ones satisfying the condition  is open dense in the $C^{\omega}$ topology of \cite{BrT86, Cl20}.   We further prove (using the analysis in \cite{PS87}) that it is generic for possibly non-convex real analytic centrally symmetric domains.

Let $\Delta_\Omega^B$ denote the Euclidean Laplacian on $\Omega$ with boundary conditions $B$ (either $B = D$, Dirichlet or $B = N$,
Neumannn).

\begin{theo} \label{ONESYM} For either  Dirichlet (or Neumann)  boundary
conditions $B$, the map $$\Omega \longmapsto \text{Spec}( \Delta_\Omega^B)$$ is one-to-one on the class $\mathcal D_L$. \end{theo}

We say that $\Omega$ is conditionally spectrally determined by $ \text{Spec}( \Delta_\Omega^B)$, meaning that it is determined by
its spectrum in the class $\mathcal  D_L$ of centrally symmetric bounded analytic domains with an invariant bouncing ball orbit of length $L$.
The only other explicit infinite dimensional classes of domains known to be conditionally determined by their spectra are  up-down symmetric 
domains with an orientation reversing isometry that reverses a  bouncing ball orbits, and the   the dihedral domains (\cite[Theorem 1.4]{Z09}). The only domains known to be unconditionally determined by their spectra 
among all smooth domains are ellipses of small eccentricity \cite{HZ19}. Triangular and trapezoidal domains (see \cite{Durso} and \cite{HLR1, HLR2}) are also spectrally determined within themselves but these classes are obviously finite dimensional. There are also some non-explicit examples of nearly circular smooth domains provided by \cite{W00, W02}, that are spectrally unique among all smooth domains. Marvizi and Melrose \cite{MM} constructed a two-parameter family of planar domains that are locally spectrally unique. The two parameter family consists of domains that are defined by some elliptic integrals but it is not known that they are ellipses.

Much of the  proof of Theorem \ref{ONESYM} is a rather straightforward modification  of the proof in  \cite{Z09} that up-down symmetric are conditionally
determined by their spectra. It uses the same analytic results and only requires a change in the algebraic analysis of the  wave trace invariants. However, it seems to us worthwhile to present what seems to be the only second explicit infinite dimensional class of  conditionally spectrally determined domains. 

Moreover, there is an important new feature in the proof, or more precisely a step which
fills in  a gap (or removes one assumption from) the inverse result for up/down symmetric domains \cite{Z09}. Namely, if the domain is represented locally as a graph $y = f(x)$,   the $y$-axis being  the bouncing ball
orbit, then it was asserted in \cite{Z09} that $f''(0)$ is a spectral invariant. This is the first step in a recursive procedure to determine
all of the Taylor coefficients of $f$ at $0$, in both the centrally symmetric and up-down symmetric inverse problems. It was implicit in \cite{Z09} that $f''(0)$ could be determined by the eigenvalues of the
Poincar\'e map of the bouncing ball orbit. However,  as  explained in Section \ref{f''0}, there exist 
two functions $f, $ resp. $g$  locally defining $\Z_2$-symmetric domains (up-down or centrally symmetric) with the same bouncing ball orbit, and for which the linear
Poincar\'e maps have  the same eigenvalues.  In this
article we prove that indeed $f''(0)$ is a spectral invariant of a $\Z_2$-symmetric domain by showing that the Maslov indices in the principal wave trace wave invariants of the bouncing ball 
orbits  of the two `dual' solutions $f, g$ are different. Therefore it is not necessary to add the assumption that  $f''(0)$ is known to
proof of Theorem \ref{ONESYM}, as is done in related work on the inverse spectral problem in \cite{CdV}.
The existence of the pair $(f, g)$  also suggests a kind of duality between the  these (germs of) domains,  whose bouncing ball orbits
have the same eigenvalues of the Poincar\'e map. It  raises the question of whether one can construct complete Taylor expansions
of $f, g$ for which the billiard maps  around the  bouncing ball orbits have exactly the  same  Birkhoff normal form.  We are currently
investigating this question  \cite{HZ21}). 
It would  follow that $f''(0)$ is a quantum Birkhoff normal form invariant which is not a classical Birkhoff normal form invariant.

\begin{rema} By an unfortunate abuse of terminology, a domain with left-right and up-down  $\Z_2 \times \Z_2$ symmetry $(x, y) \to (\pm x, \pm y)$  (i.e. the symmetries of an ellipse) was
referred to as centrally symmetric in \cite{Z09} (see Corollary 1.2.) Central symmetry refers only to the symmetry $\vec x \to - \vec x$.\end{rema}

This note was stimulated by the recent article of Bialy-Mironov \cite{BM20}. In that article, it is proved that a centrally symmetric convex plane
domain which is $C^0$ foliated in a certain neighborhood of the boundary must be an ellipse. The neighborhood is that between an invariant
curve of 4-link orbits and the boundary.  The use of the 4-link orbits is novel in inverse theory, both dynamical and Laplace spectral.  It should also be emphasized that there are no analyticity assumptions in  \cite{BM20}.

The  original idea for this article was investigate to whether  the Bialy-Mironov theorem,  in combination with the  wave trace  invariants, would  prove that a real analytic  centrally symmetric domain isospectral to an ellipse was an ellipse. In fact, this problem remains open because ellipses
do not satisfy the ``non-vertex condition'' at the endpoints, nor would an isospectral domain. It is possible that one may use fifth (or higher)
derivatives instead of third derivatives, but our attempts to do so ran into a wall of complicated formulae.  Instead,  the authors
proved Theorem \ref{ONESYM}, which does not pertain to ellipses but rather to generic centrally symmetric analytic domains. 
It would also be natural to study the wave trace invariants for the 4-link orbit of  \cite{BM20}. The wave trace invariants are expressed as
a sum over the vertices of the four link. Under central symmetry, the vertices split into two pairs and each pair contributes Taylor coefficients
of $f$ at its endpoints. But that leaves two independent sets of Taylor coefficients (one at each non-equivalent endpoint), and to date it
has not proved possible to determine the domain from this kind of data. 

\subsection{Acknowledgements} 
We thank L. Stoyanov for his advice on the genericity issues in Section \ref{ODSECT}. In the end, our proof of genericity of the conditions
for convex domains is quite different from the genericity arguments in \cite{PS87,PS17}, but in the non-convex case we have followed the
argument in \cite{PS87} to the extent possible.

\section{Wave trace invariants associated to a bouncing ball orbit}
 Let $\Omega$ be a smooth plane domain and let $\Delta_\Omega^B$ be the (positive) Laplacian on $\Omega$ with boundary condition $B$ on $\d \Omega$ with eigenvalues $\{ \lambda_j^2\}_{j=1}^\infty$.  The trace of the even wave operator is defined by
 $$ w_\Omega^B(t):= \text{Tr} \cos \left( t \sqrt{\Delta_\Omega^B} \right )  = \sum_{j=1}^{\infty} \cos (t \lambda_j ).$$ 
 The sum converges in the sense of tempered distributions. Let us define wave trace invariants associated to a simple non-degenerate periodic orbit $\gamma$ of length $L_\gamma$. Let  $\hat{\rho} \in C_0^{\infty}(L_{\gamma} - \epsilon, L_{\gamma} + \epsilon)$ be a cutoff, equal to one on an interval  $(L_{\gamma} - \epsilon/2, L_{\gamma} + \epsilon/2)$ which contains
no other lengths in Lsp$(\Omega)$ occur in its support.
Then the Fourier transform of the localized wave trace (which is the same as the trace of the smoothed resolvent)
admits an asymptotic expansion of the form (see \cite{GuMe79})
\begin{equation} \label{b_j}\int_0^{\infty} \hat{\rho}(t) e^{ikt} w_B^{\Omega}(t)\, dt  \sim  {F}_{B, \gamma}(k)  \sum_{j =
	0}^{\infty} b_{\gamma, j}  k^{-j},\;\;\; k \to
\infty,\end{equation} 
where 

\begin{itemize}
	\item ${F}_{B, \gamma}(k)$ is the {\it
		symplectic pre-factor} $${F}_{B, \gamma}(k)  =
	C_0 \; (-1)^{\epsilon_B(\gamma)} \frac{e^{ i k L_{\gamma}}
		e^{i \frac{\pi}{4} m_{\gamma}}}{\sqrt{|\det (I - P_{\gamma})|}}.$$
	
	\item $P_{\gamma}$ is the Poincar\'e map associated to $\gamma$.
	
	\item $\epsilon_B(\gamma)$ is the signed  number of intersections
	of $\gamma$ with  $\partial \Omega$ (the sign depends on the
	boundary conditions; $\pm 1$ for each bounce for Neumann/Dirichlet
	boundary conditions).
	
	\item  $m_{\gamma}$ is the Maslov  index of $\gamma$.\footnote{The term Maslov index is somewhat ambiguous here, and 
	several authors refer to $m_{\gamma}$ as the Gutzwiller-Maslov index since it is the exponent arising in the Gutzwiller-Balian-Bloch	
	trace formula  (see e.g. \cite{CRL} for a discussion valid for domains with boundary).}
	\item $C_0$ is a non-zero universal constant (e.g. factors of $2 \pi$)
	which is not necessary to know for the proof of Theorem
	\ref{ONESYM}.
	\end{itemize}

The  coefficients $\{b_{\gamma, j}\}$ are easily
related (in fact equivalent) to the wave trace coefficients $\{a_{\gamma, j}\}$ defined by the singularity expansion of the wave trace at $L_\gamma$. (see \cite{GuMe79, Z09}). We work solely with the expansion \eqref{b_j}, which we
term the `Balian-Bloch expansion' after \cite{BB72}. It is clear that the `Balian-Bloch coefficients'
$b_{\gamma, j}$ are spectral invariants and it is these invariants
we use in our inverse spectral results.

\subsection{Wave trace invariants of a bouncing ball orbit}\label{f''0}

We now focus on the case where $\gamma$ is a  bouncing ball
orbit (i.e. $2$-link periodic reflecting ray). As in the
introduction, we orient $\Omega$ so that the bouncing ball orbit
is along the $y$-axis with endpoints $A = (0,\frac{L}{2}), B = (0,
-\frac{L}{2})$ and parametrize $\partial \Omega$ near $A$ by $y =
f_+(x)$ and near $B$ by $y = f_-(x)$. At this point we do not assume the domain
is up-down symmetric or centrally symmetric. 

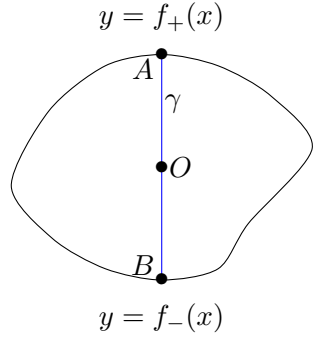
\begin{figure}
		\begin{minipage}{.5\textwidth}
	\begin{tikzpicture}[scale=0.50]
	\draw [black] plot [smooth cycle] coordinates {(4,0.5) (2.9,1.9) (1.5, 2.7) (0, 3) (-1.5,2.7)  (-3,1.5) (-4,-0.5) (-2.9,-1.9) (-1.5, -2.7) (0, -3) (1.5, -2.7) (2.3,-1.5)};
\draw[blue] (0,3)--(0,-3);
\node at (.3, 1.7) [align=center]{$\gamma$};
\node at (0, 4) [align=center]{$y = f_+(x)$};
\node at (0, - 4) [align=center]{$y = f_-(x)$};
\draw (0,3) node{$\bullet$};
\draw (0,-3) node{$\bullet$};
\draw (0,0) node{$\bullet$};
\node at (.5, 0 ) [align=center]{$O$};
\node at (-.5, 2.6) [align=center]{$A$};
\node at (-.5, -2.6) [align=center]{$B$};
	\end{tikzpicture}
	\centering
	\caption{Local defining functions near a bouncing ball orbit of a general simply-connected smooth domain.}
\end{minipage}
\end{figure}

When $\gamma$ is elliptic, the
eigenvalues of $P_{\gamma^r}$ are of the form $\{e^{\pm i r \alpha}\}$
($0< \alpha  \leq \pi $) while in the hyperbolic case they are of the
form $\{e^{\pm r \alpha}\}$ ($\alpha >0)$. Thus
\begin{equation} \label{det}
\det (I - P_{\gamma^r})   = \begin{cases} 2 - 2 \cos (r \alpha) & \text{(elliptic case)},
\\  2 - 2 \cosh (r \alpha) & \text{(hyperbolic case)}. \end{cases}
 \end{equation}
When the domain is up-down or centrally symmetric one has $f''_+(0) =  -f_-''(0)$, and one can determine $\alpha$ in terms of $f_+''(0)$ (and vice versa)  via (see \cite{KT,K01})
\begin{equation} \label{quadratic equation} \left ( 1+ L f''_+(0) \right )^2   = \begin{cases}   \cos^2 (\alpha/2) & \text{(elliptic case)}, \\
 \cosh^2(\alpha/2)   & \text{(hyperbolic case)}. \end{cases} \end{equation}
 If we fix $\alpha$, each of these is a quadratic equation in terms of $f''_+(0)$ and its two roots are given by
 \begin{equation} \label{roots of quadratic}  f''_+(0)  = \begin{cases}  \frac{1}{L} \left (-1 \pm  \cos (\alpha/2) \right ) & \text{(elliptic case)}, \\
\frac{1}{L} \left ( -1 \pm  \cosh (\alpha/2) \right )   & \text{(hyperbolic case)}. \end{cases} \end{equation}
 
 We note that (\cite{KT,K01})
 \begin{equation*} \begin{cases} \gamma \; \text{is elliptic} & \text{if and only if}  \;  -2<Lf_+''(0) <0, \\
\gamma \;  \text{is hyperbolic}  & \text{if and only if}\; Lf_+''(0) >0 \; \text{or} \; Lf_+''(0) < -2. \end{cases} \end{equation*}

We  define the  length functionals in Cartesian coordinates for
the two possible orientations of the  $r$th iterate of a bouncing
ball orbit by
\begin{equation} \label{LPM} {\mathcal L}_{\pm, 2r } (x_1, \dots, x_{2r}) = \sum_{p = 1}^{2r} \sqrt{(x_{p + 1} - x_p)^2 +
	(f_{w_{\pm}(p + 1) }(x_{p +1}) - f_{w_{\pm}(p)}(x_p))^2}. \end{equation}
Here, $w_{\pm}: \Z_{2r} \to \{\pm\}$, where $w_{+}(p)$ (resp. $w_-(p))$ alternates sign starting with
$w_+(1) = +$ (resp. $w_-(1) = -$). Also, we use cyclic index
notation where $x_{2r + 1} = x_1$.

We will  need formulae for the entries of the inverse  Hessian of
${\mathcal L}_{+, 2r}$ at its critical point $(x_1, \dots, x_{2r}) =
0$ in Cartesian coordinates corresponding to the $r$th repetition
of a bouncing ball orbit. We denote
$$ H_{2r} = \text{Hess} \, \mathcal {\mathcal L}_{+, 2r}(0).$$ 
The following lemma was proved in \cite{Z09}.
\begin{lemm} \label{hpq}Suppose $f''_+(0) =  -f_-''(0)$, that is the curvatures at the two end points of the bouncing ball orbit $\gamma$ are the same.  Let $h_{2r}^{pq}$ be the matrix elements of the inverse matrix
$H_{2r}^{-1}$ and let 
$$a:= -2(1 + Lf''_+(0)).$$  Then  for  $1 \leq p \leq q \leq 2r$, 
	$$\begin{array}{l}  h^{pq}_{2r} = \frac{L}{2 \left ( T_{2r}(- a/2) -1 \right )} \left ( U_{2r - q + p -1}(-a/2) +
	U_{q - p - 1}(-a/2)\right),  \end{array}$$
	where  $T_n,$ resp. $U_n$, are Chebychev
	polynomials  of the first, resp. second, kind.
	They are defined by:
	$$T_n(\cos \theta) = \cos n \theta,\;\;\;\;\; U_n(\cos \theta) = \frac{\sin (n + 1) \theta}{\sin \theta}.$$
\end{lemm}
As a result, for the two choices of $f''(0)$ in terms of $\alpha$ (See \eqref{roots of quadratic}), we get that 
\begin{itemize}
	\item If $1 + L f_+''(0) = \cos(\alpha /2) ( \text{or}\; \cosh (\alpha/2) \; \text{in the hyperbolic case})$, then 
\begin{equation*}  \label{hpq+}  
h_{2r}^{pq}
=  \begin{cases} - \left ( \frac{L}{ 2\sin (\alpha/2)} \right ) \frac{\cos \left ( {( r - q + p ) \alpha/2} \right ) } {\sin (r \alpha/2)}  & \text{(elliptic case)} \\  \\
\left ( \frac{L}{ 2\sinh (\alpha/2)} \right ) \frac{\cosh \left ( {( r - q + p ) \alpha/2} \right ) } {\sinh (r \alpha/2)} & \text{(hyperbolic case)}. \end{cases}
\end{equation*}
	\item If $1 + L f_+''(0) = - \cos(\alpha /2) ( \text{or}\; - \cosh (\alpha/2) \; \text{in the hyperbolic case})$, then 
\begin{equation*}  \label{hpq-}  
h_{2r}^{pq}
=  \begin{cases} - \left ( \frac{L  (-1)^{p - q}}{ 2\sin (\alpha/2)} \right ) \frac{\cos \left ( {( r - q + p ) \alpha/2} \right ) } {\sin (r \alpha/2)}  & \text{(elliptic case)} \\  \\
\left ( \frac{L  (-1)^{p - q}}{ 2\sinh (\alpha/2)} \right ) \frac{\cosh \left ( {( r - q + p ) \alpha/2} \right ) } {\sinh (r \alpha/2)} & \text{(hyperbolic case)}. \end{cases}
\end{equation*}
\end{itemize}
We observe that, up to multiplication by $(-1)^{p-q}$, the expressions \eqref{hpq+} and \eqref{hpq-} are identical.
It is also obvious from Lemma \ref{hpq} that,
\begin{coro} If $f''_+(0) =  -f_-''(0)$, then for all $1\leq q \leq 2r$: $$h_{2r}^{qq} = h_{2r}^{11}.$$
\end{coro}
In general and without the assumption $f''_+(0) =  -f_-''(0)$, we know that $h_{2r}^{qq} = h_{2r}^{11}$ for $q$ odd, and $h_{2r}^{qq} = h_{2r}^{22}$ for $q$ even.

The following expressions for the wave invariants $\{ b_{\gamma^r, j} \}$  were found by the second author \cite{Z09} (denoted by $B_{\gamma, j}$ in that article). Although the main result of \cite{Z09} is concerned with up-down symmetric domains but the following result was proved 
for general smooth domains with a non-degenerate bouncing ball orbit. It does not assume that $f''_+(0) =  -f_-''(0)$.

\begin{theo} \label{b_j expression} Let $\Omega$ be a smooth domain with a bouncing
	ball orbit   $\gamma$  of length $ L_{\gamma}$. Let $r \in \mathbb N$ and assume $\gamma^r$ is non-degenerate. Then for $j >1$:
	\begin{equation*}
	 \begin{array}{l}  b_{\gamma^r, j - 1}  = 4 r L {\mathcal A}_0(r)  \Big( 2 r w_{{\mathcal G}_{1,
				j}^{2j, 0}} ((h^{11}_{2r} )^j f^{(2j)}_{+}(0) - (h^{22}_{
			2r})^j f^{(2j)}_{-}(0))  \\ \\ + 4 {\displaystyle \sum}_{q, p = 1 }^{2r}
	 \left (w_{{\mathcal G}_{2,  j + 1 }^{2j - 1, 3,
	 		0}}  (h^{pp}_{2r})^{j
			- 1} h^{qq}_{2r} h^{pq}_{2r} + (w_{\widehat{{\mathcal
					G}}_{2, j + 1 }^{2j - 1, 3,
				0}}) (h^{pp}_{2r})^{j - 2}
		(h^{pq}_{2r})^3 \right ) w_+(p) w_+(q) f_{w_+(p)}^{(2j - 1)}(0)
		f_{w_+(q)}^{(3)}(0) \Big)
		\\ \\ + R_{2r} ({\mathcal
			J}^{2j - 2} f_+(0), {\mathcal J}^{2j - 2} f_-(0)),\end{array} 	\end{equation*}
		where the remainder $ R_{2r} ({\mathcal J}^{2j - 2} f_+(0),
		{\mathcal J}^{2j - 2} f_-(0))$ is a polynomial in the designated
		jet of $f_{\pm}.$ Here, $w_+(p) = (-1)^{p+1}$ and $w_{\gcal} = \frac{1}{|Aut (\gcal)|}$ are
		combinatorial factors independent of $\Omega$ and $r$. The function ${\mathcal A}_0(r)$ is a non-zero function of $r$ and is independent of $\Omega$ and $j$.
	\end{theo}
For the sake of simplicity of our notations, from now on we set:
$$ \tilde{C}_j:= w_{{\mathcal G}_{1,
		j}^{2j, 0}}, \quad {C}_j:= w_{{\mathcal G}_{2,  j + 1 }^{2j - 1, 3,
		0}} , \quad  \hat{C}_j:=  w_{\widehat{{\mathcal
			G}}_{2, j + 1 }^{2j - 1, 3,
		0}}.$$ 
		
		The following proposition will become very useful in the next section. It gives a description of the Maslov index in terms of the signature of the Hessian of the length functions $\mathcal L_+$.
		
		\begin{prop} \label{maslov} The Maslov index $m_{\gamma^r}$ associated to $\gamma^r$, the $r$-th iteration of a non-degenerate bouncing ball orbit $\gamma$, appearing in the prefactor of the wave trace expansion, is given by 
			$$m_{\gamma^r} = \ell_{r} + \text{sgn} (\text{Hess} \, \mathcal L _+) \qquad (\text{mod}\; 8),$$
where $\ell_{r}$ is an integer that depends only on $2r$ and is independent of $\gamma^{r}$. 
 		\end{prop}
 	\begin{proof} This identity, although not explicitly stated, is implicit in the proof of Lemma 2.3 of \cite{Z09} (See page 233). In fact it reveals that the Maslov factor in the prefactor is given by 
 	$$e^{i \pi \, m_{\gamma^r}/4} = (-i)^{2r} \left ( e^{-3\pi i /4}\right )^{2r} e^{ i \pi  \, \text{sgn} (\text{Hess} \, \mathcal L _+) /4} .$$ 

 Hence, in fact in Proposition \ref{maslov} we have $\ell_ r$.   \end{proof}
\begin{rema}  The Maslov index above combines phases from three sources: The signature of the Hessian (which of course arises
from stationary phase), and  from phases in the Dirichlet (resp. Neumann) Green's function. See \cite{BB72, CRL, Z09} for further
discussion. 

From the above discussion one can actually deduce that $m_{\gamma^r}$ is an even integer. This is consistent with the statement of Guillemin-Melrose \cite{GuMe79} where the Maslov factor is in the form $i^{ m'}$. To see this, let $n_+$ and $n_-$ be the number of positive and  negative eigenvalues of $\text{Hess} \, \mathcal L _+ $, respectively. Since $n_+ = 2r - n_-$, we obtain that 
\begin{align*} e^{i \pi \, m_{\gamma^r}/4} & = (-i)^{2r} \left ( e^{-3\pi i /4}\right )^{2r} e^{ i (2r - 2n_-) \pi /4}  \\
                                                                           & =  e^{6r i \pi /4}  e^{- 2r i \pi /4} e^{- 2i n_- \pi /4}\\        
                                                                            & = i^{2r - n_-}  \\
                                                                            & = i^{n_+}.\end{align*}
                                                                            
Therefore, it appears that the Maslov index in  the trace asymptotic of Guillemin-Melrose \cite{GuMe79}  is given by 
$$m'_{\gamma^r}= \frac{m_{\gamma^r}}{2} =  n_+ \qquad (\text{mod} \; 4).$$ Note that the above calculation can be imitated, thus is also valid, for any non-degenerate periodic orbit and not only for iterations of a bouncing ball orbit. 



\end{rema}

\section{Centrally symmetric plane domains; Proof of Theorem \ref{ONESYM}}

We recall that a centrally symmetric domain contains the origin $O=(0, 0)$, and has a central symmetry $\sigma(x, y) = (-x, -y)$ fixing $O$. We have the following simple lemma. 

\begin{lemm}
Every smooth simply connected centrally symmetric domain $\Omega$ has at least one bouncing ball orbit that goes through $O$. If in addition $\Omega$ is star-shaped about the origin $O$, then $\Omega$ has at least two such bouncing ball orbits. 
\end{lemm}
\begin{figure}[H]
	\begin{minipage}{.5\textwidth}
		\begin{tikzpicture}[scale=0.50]
		
		\draw [black] plot [smooth cycle] coordinates {(2, 0) (2, -2) ( 1, -3) (-2, -2) (-2, -1) (.5, -1.5) (.5, -1) (-2, 0) (-2, 2) (-1, 3) (2,2)(2, 1) (-.5, 1.5) (-.5, 1)};
		\draw[blue] (1.2,-3)--(-1.2,3);
		\draw[red] (.19,.79)--(-0.19,-.79);
		\node[red] at (-.5, -.3) [align=center]{$\gamma$};
		
		\draw (0,0) node{$\bullet$};
		\node at (.5, 0 ) [align=center]{$O$};
		\end{tikzpicture}
		\centering
		\caption{A centrally symmetric domain that is not star-shaped. The diameter is not a bouncing ball orbit but the minimizer of $D(P)$ provides a bouncing ball orbit. }
	\end{minipage}%
	\begin{minipage}{.5\textwidth}
		\begin{tikzpicture}[scale=0.50]
		\draw [black] plot [smooth cycle] coordinates {(4,0.5) (2.9,1.9) (1.5, 2.7) (0, 3) (-1.5,2.7)  (-2,1.5) (-4,-0.5) (-2.9,-1.9) (-1.5, -2.7) (0, -3) (1.5, -2.7) (2,-1.5)};
		\draw[red] (0,3)--(0,-3);
		\draw[blue] (4, .66)--(-4,-.66);
		\draw (0,0) node{$\bullet$};
		\node at (.5, .6 ) [align=center]{$O$};
		\end{tikzpicture}
		\centering
		\caption{A star-shaped centrally symmetric domain always has at least two bouncing ball orbits.}
	\end{minipage}
\end{figure}
\begin{proof} First let us assume that $\Omega$ is star-shaped about $O$. Thus for each $P$ on $\d \Omega$ the line segment $\overline{(-P)P}$ lies inside $\Omega$.  Consider the  maximum and minimum points of $D(P) = d(P, -P)^2$ on $\partial \Omega$. Of course, $D(P) = 4 |P|^2$. Its critical
points occur when $ \langle c(0), \dot{c}(0) \rangle = 0$ where $c(t)$ is an arc-length parameterization of $\partial \Omega$ with  $c(0)= P$. Thus, the radial line $\overline{0 P}$ from $0$ to $P$ is orthogonal  to $\partial \Omega$ at $P$. Since $D(P)$ is symmetric
in $P, -P$ the same statements hold for $-P$, and therefore the line segment  $\overline{(-P) P}$ is the trace of a  bouncing ball orbit, denoted by $\gamma$. 
That is, $\gamma$ is the orbit in $T^*\Omega$ which projects to $\overline{(-P) P}$; it shuttles back and forth between $P$ and $-P$, covering
the segment twice.

Now assume $\Omega$ is not necessarily star-shaped about $O$. Since $\Omega$ is simply connected and centrally symmetric it must contain $O$ (by the Brouwer fixed point theorem). Let $P$ be the minimizer of $D(P)$ and consider the segment $\overline{(-P) P}$ (which apriori
could go outside of $\Omega$) . Since the segment contains $O$,  it must have a non-empty intersection with $\Omega$. If the entire segment does not lie inside $\Omega$ we get a contradiction that it is a minimizer of $D(P)$, because otherwise the maximal connected part of the segment that contains $O$ and intersects $\Omega$ would have a smaller length than $\overline{(-P) P}$. 
\end{proof}

\begin{rema} 
It is proved in \cite[Corollary 1.3]{Gh04} that if $\Omega$ is centrally symmetric and convex, then each of the shortest periodic billiard trajectories in $\Omega$ is a bouncing ball orbit. However, we do not assume convexity in our statements (except for Section \ref{CONVEX}), yet the central symmetry provides the existence of at least one bouncing ball orbit. 
\end{rema}

Among the bouncing ball orbits, we pick one which satisfies the assumptions of $\dcal_{L}$ (where $2L$ is the length of the orbit), and denote
it by $\gamma.$  As in the introduction, since $0 \in \gamma$,  we may rotate  $\Omega$ to make $\gamma$ the vertical  `$y$'-axis.  The orthogonal line through $0$ is also $\sigma$-invariant and we refer to it as the $x$-axis.

Then  locally near the vertices of $\gamma$, $\partial \Omega$ consists of two graphs, $$\partial \Omega = \{y = f_+(x)\} \cup \{y = f_-(x)\}, $$ lying above, resp. below, the horizontal axis. 
If $F(x,y) = 0$ is a defining function of $\partial \Omega$, then
the central symmetry implies that $F(-x, -y) = 0 \iff F(x,y) = 0$. Hence, if $y = f_+(x)$ then $-y = f_-(-x),$ so that
\begin{equation} f_-(-x) = - f_+(x). \end{equation}

\subsection{Wave invariants of centrally symmetric domain}

We shall follow the results in \cite{Z09} regarding the expressions for the wave trace invariants of $\Omega$, which we reviewed in the previous section. The centrally symmetric domain of this article is not `up-down' symmetric across the axis $\gamma$,
which would be the symmetry $f_-(x) = - f_+(x)$ but rather $f_-(-x) = - f_+(x)$. We now consider how this impacts the expressions for the wave invariants of a centrally symmetric domain. First we observe that for $k =0, 1, 2, \dots$, we have
\begin{equation} \label{taylor coeffs relations}
f^{(2k)} _-(0) = -  f^{(2k)}_+(0), \quad 
f^{(2k+1 )} _-(0) =   f^{(2k+1)}_+(0)
\end{equation}

\begin{rema} Comparing $k$th derivatives in the cases $f_-(x) = - f_+(x)$, resp.  $f_-(-x) = - f_+(x)$, the only difference is the extra factor of $(-1)^k$
in the centrally symmetric case. Hence, the even derivatives are identical and odd derivatives differ by a factor of $-1$. In the next Proposition, the
universal formulae for top derivative terms of wave invariants of bouncing ball orbits at a given order are exactly the same for
the up-down and the centrally symmetric cases.   \end{rema}

Using the fact that $h_{2r}^{pq}$ is a function of $q-p$ (see Lemma \ref{hpq}), and applying the relations \eqref{taylor coeffs relations} to Theorem \ref{b_j expression}, we obtain:
 \begin{coro} \label{b_j corollary} The wave trace invariant for $\gamma^r$ are given by
$$ \begin{array}{lll} b_{\gamma^r, j -
 1} & = &    4 L  r {\mathcal A}_r(0) \Big ( 2 r  \widetilde{C}_j \;
 (h_{2r}^{11})^j \; f_+^{(2j)}(0) \\ && \\ && + 8 r  C_j \;
 (h_{2r}^{11})^j \sum_{q = 1}^{2r} h_{2r}^{1q} \; f_+^{(3)}(0)  f_+^{(2j - 1)}(0) \\ && \\
 && +  8r \widehat{C}_j \; (h_{2r}^{11})^{j - 2}
 \sum_{q = 1}^{2r} (h_{2r}^{1q})^3 \; f_+^{(3)}(0)  f_+^{(2j - 1)}(0)\Big ) \\ && \\
 && +  R_{2r} ({\mathcal J}^{2j - 2} f_+(0)).\end{array}  $$
  \end{coro}
All the notations in the above expression were introduced in Theorem \ref{b_j expression}. We recall that $\widetilde{C}_j$,  ${C}_j$, and  $\widehat{C}_j$ are non-zero positive constants that depend only on $j$.  See the precise definition after Theorem \ref{b_j expression}.

\textbf{Proof of the main theorem}. Assume that $\Omega_f$ and $\Omega_g$ are two isospectral domains in the class $\mathcal D_L$ where  we have denoted their top defining functions by $f$ and $g$, respectively. We want to show that locally either $f(x) =g(x)$ or $f(x) = g(-x)$. Since $f$ and $g$ are analytic it suffices to prove that their Taylor coefficients agree accordingly. By assumptions, $f(0)= g(0) = L/2$ and $f'(0) = g'(0)=0$. Let us show that $f''(0) = g''(0)$. Since the domains are isospectral they must have the same trace prefactor, in particular, the same $|\det ( I - P_{\gamma^r})|$ and the same Maslov index $m_{\gamma^r}$ for $r=1$ and $r=2$. The following was proved in \cite{Z09} (see Prop 6.2):
\begin{prop} Suppose $f_+''(0) = - f''_-(0)$. Let $a = -2(1+Lf_+''(0))$. Then the eigenvalues of Hess$(\mathcal L_+)(0)$ are given by
	$$ a + 2 \cos (k \pi /r); \quad k=0, \dots, {2r-1} .$$
	\end{prop}
By this proposition, when $r=1$, the eigenvalues are $a+2$ and $a-2$. Thus,
 \begin{equation} \label{sgn}  \text{sgn Hess} ( \mathcal L_+) = \begin{cases}  0  & \text{(elliptic case)}, \\  2 & \text{(hyperbolic case and $a>2$)}, \\
 -2 & \text{(hyperbolic case and $a <-2$)}. \end{cases} \end{equation}
 This shows that the bouncing ball orbits $\gamma_f$ and $\gamma_g$ of the isospectral domains $\Omega_f$ and $\Omega_g$ are both elliptic or both hyperbolic.  On the other hand by \eqref{det},
 $$ | \det (I - P_{\gamma}) |  = \begin{cases} 2 - 2 \cos (\alpha) & \text{(elliptic case)},
 \\  2 \cosh (\alpha) -2 & \text{(hyperbolic case)}. \end{cases}$$
 Hence if the bouncing ball orbits are elliptic, they must have the same $\alpha \in (0, \pi]$, and if they are hyperbolic then they have the same $\alpha \in (0, \infty)$. However, $\alpha$ does not uniquely determine the curvature at the vertices of the bouncing ball orbit. We recall that by \eqref{quadratic equation}, given $\alpha$, there are two possible values for $f_+''(0)$ that satisfy:
 $$f''_+(0)  = \begin{cases}  \frac{1}{L} \left (-1 \pm  \cos (\alpha/2) \right ) & \text{(elliptic case)}, \\
 	\frac{1}{L} \left ( -1 \pm  \cosh (\alpha/2) \right )   & \text{(hyperbolic case)}. \end{cases}$$
 In order to prove that $f''(0) = g''(0)$, we need to show that the following possibilities cannot happen for isospectral domains $\Omega_f$ and $\Omega_g$:
 In the elliptic case,
 $$ f''(0)= \frac{1}{L} \left (-1 +  \cos (\alpha/2) \right ), \quad g''(0)=  \frac{1}{L} \left (-1 -  \cos (\alpha/2) \right ) $$
 and in the hyperbolic case,
 $$ f''(0)= \frac{1}{L} \left ( -1 +  \cosh (\alpha/2) \right ), \quad g''(0)=  \frac{1}{L} \left ( -1 -  \cosh (\alpha/2) \right ) $$
 To rule out these cases we use the Maslov index again. Note that for $r=1$ and when both orbits $\gamma_f$ and $\gamma_g$ are elliptic, by \eqref{sgn} the signature of the Hessian is zero so we shall use $r=2$ instead. Indeed in this case, the eigenvalues of the Hessian are $a+2, a, a-2, a$ and as a result,
 \begin{equation*}  \text{sgn Hess} ( \mathcal L_+) = \begin{cases}  -2  & \text{for $\Omega_f$}, \\  2 & \text{for $\Omega_ g$.} \end{cases} \end{equation*}
 In the hyperbolic case, $r=1$ distinguishes the two domains because in this case the eigenvalues are $a+2$ and $a-2$, hence
  \begin{equation*} \text{sgn Hess} ( \mathcal L_+) = \begin{cases}  -2  & \text{for $\Omega_f$}, \\  2 & \text{for $\Omega_g$.} \end{cases} \end{equation*}
  This concludes our proof of $f''(0) = g''(0)$, which was not proved correctly in \cite{Z09} and is a new ingredient of the current article. In particular this shows that the two isospectral domains have the same inverse Hessian coefficients $h^{2r}_{pq}$.

 The rest of the proof is identical to \cite{Z09}. For the sake of completeness we provide it here and we also provide the details of the poof in the hyperbolic case which was omitted in \cite{Z09}. 

The plan is to use the expressions for the wave invariants presented in Corollary \ref{b_j corollary} and argue inductively on $j$ that $f^{2j}(0)$ and $f^{(2j - 1)}(0)$ are wave trace invariants, hence spectral invariants of the Laplacian among
 domains in ${\mathcal D}_{L}$. First we simplify the expression in the corollary using the following formula which was proved in \cite{Z09} (see Prop 6.5):
 \begin{prop} For all $r \in \mathbb N$, 
 	$$ \sum^{2r}_{q=1} h_{2r}^{1q} = - \frac{L}{a+2}.$$
 	\end{prop}

 Dividing $b_{\gamma^r, j-1}$ from Cor \ref{b_j corollary} by the spectral invariant $8 L  r^2 {\mathcal A}_r(0)  (h_{2r}^{11})^{j-2}$, and using the above proposition, it follows that:
 $$ \begin{array}{lll} b'_{\gamma^r, j -
 	1}& := &     (h_{2r}^{11})^2  \Big (   \widetilde{C}_j  f_+^{(2j)}(0) - \frac{4L}{a+2}  C_j f_+^{(3)}(0)  f_+^{(2j - 1)}(0) \Big )\\ && \\
 && +  \sum_{q = 1}^{2r} (h_{2r}^{1q})^3 \Big ( 4 \widehat{C}_j \;
  \; f_+^{(3)}(0)  f_+^{(2j - 1)}(0) \Big ) \\&& \\
  && +  R_{2r} ({\mathcal J}^{2j - 2} f_+(0)), \end{array}  $$
 is a spectral invariant for each $j$. 
 We claim that the expressions in parentheses can be decoupled by proving that $\sum_{q = 1}^{2r} (h_{2r}^{1q})^3$ and  $(h_{2r}^{11})^2$ are linearly independent as function of $r$. It suffices to prove that 
 $$G(r):= \frac{ \sum_{q = 1}^{2r} (h_{2r}^{1q})^3}{(h_{2r}^{11})^2},$$
 is non-constant in $r$, for $r=1, 2$. In fact $G(1) = G(2)$ is satisfied only for certain values of $a$, which forms the `bad set' $\mathcal B$ that we would like to exclude. To find these values of $a$, we write out the equation $G(1)= G(2)$ using  Lemma \ref{hpq} and the formulas for Chebyshev polynomials $T_1, \dots, T_4$ and $U_1, \dots, U_4$, to obtain:
 $$\begin{array}{l} \frac{a^3 - 8}{(a^2 - 4)^3} \frac{(a^2 -
 	4)^2}{a^2} = \frac{(a^4 - 4a^2)^2}{(a^3 - 2a)^2} \frac{a^9 - 6 a^7
 	- 2a^6 + 12 a^5}{(a^4 - 4a^2)^3} \\ \\
 \iff (a^3 - 2a)^2 (a^3 - 8) = a^9 - 6 a^7 - 2a^6 + 12 a^5 .
 \end{array}$$
 A little bit of cancellation reduces the  equation to a degree $6$ polynomial whose  distinct roots are  $\{0, -1, 2, -2\}$. 
 
 Since by \eqref{quadratic equation},
 $$|a| = 2|1+ L f''_+(0)|   = \begin{cases}   2\cos (\alpha/2) & \text{(elliptic case)}, \\
 2\cosh(\alpha/2)   & \text{(hyperbolic case)}, \end{cases} $$ 
 in the elliptic case we require that $\cos (\alpha /2) \notin \{0, \frac12, 1\}$, and in the hperbolic case we only require the non-degeneracy assumption $\alpha \neq 0$. These are precisely the conditions we imposed on the class $\mathcal D_L$. 
 
 We start the final argument by letting $j=2$ in the formula for $b'_{\gamma^r, j -
 	1}$. By the independence of $\sum_{q = 1}^{2r} (h_{2r}^{1q})^3$ and  $(h_{2r}^{11})^2$,  we get that $(f_+^{(3)}(0))^2$ is a spectral
 invariant, and we may assume with no loss of
generality that $f^{(3)}(0)
> 0$ or otherwise we can reflect the domain about the $y$ axis to obtain this condition. It then  follows that $f^{(4)}(0)$ is determined. Arguing by induction from  $j \to j + 1$, we assume that the $(2j-2)$ jet  ${\mathcal J}^{2j -2} f_+(0)$ of $f_+$ at $0$ is
 known. The lower order derivative  terms denoted by $R_{2r} {\mathcal J}^{2j -2} f_+(0)$  are universal polynomials in the data ${\mathcal J}^{2j -2}
 f_+(0)$, hence are known by the induction hypothesis. Thus, it suffices
to determine $f_+^{(2j)}(0)$ and $ f_+^{(2j - 1)}(0)$.
 By the decoupling argument, we can
determine $f_+^{(3)}(0)f_+^{(2j - 1)}(0)$, hence $f_+^{(2j -
1)}(0)$, as long as $f_+^{(3)}(0) \not= 0.$ But then we can
determine $f_+^{(2j)}(0).$ By induction, $f_+$ is determined (i .e. $f=g$) and hence $\Omega_f = \Omega_g$, completing the proof of Theorem \ref{ONESYM}.

\section{\label{ODSECT} Open denseness of the conditions} 
In this section, we prove that the class $\dcal_L$ is  residual in the class of simply connected real analytic $\sigma$-invariant domains. In other words the conditions (1)-(4) in the definition of  $\dcal_L$ are generic. We also prove if we only restrict ourselves to strictly convex domains, then these conditions are  open dense. The relevant topology is
the analytic topology, defined below. Although there exist numerous studies of generic properties of maps and domains, they almost
always refer to the $C^{\infty}$ topology. The most relevant to this article are \cite{PS87,PS17}, which have several results on  generic
properties of $C^{\infty}$ billiards. Only a few articles, to our knowledge, study generic properties of analytic maps or domains, and none
seem to apply directly to billiards.  We refer
to \cite[Section 2]{BrT86} and to \cite{Cl20}  for some results on other geometric problems, but which nevertheless useful guidelines.

Let us first define the $C^\infty$ and $C^\omega$ topologies. We denote $S^1 = \R / \Z$ and 
$$ C^{\infty}(S^1, \R^2) = \{ \alpha: S^1 \to \R^2, \alpha  \; \text{smooth}\}. $$
A basis for the $C^\infty$ topology is defined by
$$B_{\beta_0, \ep_0, m} = \{ \alpha \in C^{\infty}(S^1, \R^2): \; \| \alpha - \beta_0 \|_{C^m} < \ep_0 \}, $$
where $\beta_0 \in C^{\infty}(S^1, \R^2)$, $\ep_0>0$, and $m \in \mathbb N$. We note that since $S^1$ is compact the notions of \textit{weak} and \textit{strong} $C^\infty$ topologies are identical.  In this topology $\alpha_n \to \alpha_0$ if for all $m$, $\|\alpha_n -\alpha_0 \|_{C^m}\to 0$. 

We also denote
$$C^{\omega}(S^1, \R^2) = \{ \alpha: S^1 \to \R^2, \alpha  \; \text{analytic}\}, $$
to be the class of analytic maps. For $\alpha \in C^{\omega}(S^1, \R^2)$ we denote $\tau(\alpha)$ to be width of the maximal cylinder $T_\tau=S^1 \times (-\tau, \tau)$ on which $\alpha$ accepts a holomorphic extension which we denote by $\tilde \alpha$.  
Then the the $C^\omega$ topology is defined by the basis
$$B_{\beta_0, \ep_0, \tau} = \{ \alpha \in C^{\omega}(S^1, \R^2): \; \tau(\alpha) > \tau, \; \| \tilde \alpha - \tilde \beta_0 \|_{C^0(T_\tau)} < \ep_0 \}, $$
where $\beta_0 \in C^{\omega}(S^1, \R^2)$, $ \tau(\beta_0) > \tau >0$, and $\ep_0>0$.  
In the analytic  topology $\alpha_n \to \alpha_0$ if  $\|\tilde\alpha_n - \tilde \alpha_0 \|_{C^0(T_{\tau_n})}\to 0$ for every $\tau_n< {\min (\tau(\alpha_n), \tau(\alpha) )}$. 

We refer to  \cite[Section 2.6]{KrP}, \cite[Section 2]{BrT86} and \cite{Cl20}  for expositions of the analytic topology on the space of real analytic functions on a connected open set of $\R^n$. Equipped with this topology, the space of real analytic functions is a Baire space, i.e. residual sets are dense.

To compare the two topologies we note that if $U$ is open in $C^\infty$ then $U \cap C^\omega$ is open in $C^\omega$. This is because the analytic topology is a stronger topology (involves more constraints to belong to an open set). This observation is used in \cite{BrT86, Cl20, KM} and we refer there for further discussion. One can also easily observe that $C^\omega$ is dense in $C^\infty$. To see this,  suppose $\alpha \in C^\infty$ and consider its Fourier series written as 
\begin{equation} \label{FS1}  \alpha (\theta) = (\alpha^{(1)}(\theta) , \alpha^{(2)}(\theta)) = \left ( \sum a^{(1)}_j e^{2ij \pi \theta }, \sum a^{(2)}_j e^{2ij \pi \theta } \right ). \end{equation}
Then for example, the sequence $\alpha_n$ defined by 
$$\alpha_n (\theta )= \left ( \sum a^{(1)}_j e^{-\frac{2 \pi | j |}{n}} e^{2ij \pi \theta}, \sum a^{(2)}_j e^{-\frac{2 \pi  |j|}{n}} e^{2ij \pi \theta}  \right ), $$
is analytic  (as it accepts a holomorphic extension to the cylinder $T_{1/n}$) and it
converges to $\alpha$ in the $C^\infty$ topology.  

Now we present:
\begin{def1}\label{def1}
	Let $L$ and $P$ positive be fixed. We define $\mathcal E^\infty_{\sigma, L, P} $ to  be the subclass of $C^{\infty}(S^1, \R^2)$ of smooth embeddings $\alpha$ into $\R^2$ with perimeter $P$,  such that the image of $\alpha$ is $\sigma$-invariant and it has a $\sigma$-invariant bouncing ball orbit of length $L$ aligned on the $y$-axis. It is obvious that we always have $P>2L$.  The $C^\infty$ topology on $ \mathcal E^\infty_{\sigma, L, P} $ is inherited from the topology on $C^\infty(S^1, \R^2)$. Analogously,  we denote 
	$$\mathcal E^\omega_{\sigma, L, P}  = \mathcal E^\infty_{\sigma, L, P}  \cap C^{\omega}(S^1, \R^2), $$
	and equip it with the analytic topology. 
\end{def1}

$\mathcal{E}^{\omega}_{\sigma, L, P} (S^1, \R^2)$ is an infinite dimensional real analytic manifold (see \cite[Section 8]{KM}, especially \cite[Theorems 8.2-8.3]{KM} for a systematic account). 
A tangent vector at $\alpha \in \mathcal{E}^{\omega}_{\sigma, L, P}$ is a vector field $X$ along $\alpha$ such that
$$ X(\alpha(\theta)) = \frac{d}{d \epsilon} \alpha_{\epsilon}( \theta )  |_{\epsilon = 0} , \;\; \theta \in {S^1}, $$ for some curve $\epsilon \to \alpha_\epsilon$, $\epsilon \in (- \epsilon_0, \epsilon_0)$, in $\mathcal{E}^{\omega}_{\sigma, L, P} (S^1, \R^2)$ satisfying $\alpha_0 = \alpha$.  Given a function $f$ on $\mathcal{E}^{\omega}_{\sigma, L, P} (S^1, \R^2)$, we denote $df|_{\alpha}$ to be the linear functional on $T_{\alpha} \mathcal{E}^{\omega}_{\sigma, L, P}$ defined by
$$df|_{\alpha} (X) = \frac{d}{d \epsilon} f ( \alpha_{\epsilon} )  |_{\epsilon = 0}. $$ Clearly by our definition $X$ must be $\sigma$-invariant because $\alpha_\epsilon$ is $\sigma$-invariant. The vector field $X$ must also satisfy the constraints that 
\begin{equation} \label{P and L preserved} d P(X) = 0 = d L(X). \end{equation}
We view $X$ as an `infinitesimal deformation preserving the perimeter $P$ and $L$.  To avoid redundancies, we assume that $X(\alpha(\theta)) \bot \alpha'(\theta)$
is normal to $\alpha$ and write $X(\alpha(\theta)) = \dot{\rho}(\alpha(\theta)) \nu (\alpha(\theta))$ for some analytic function $\dot{\rho}$ where $\nu$ is the unit inward normal. 
Here we assume that $\alpha_{\epsilon}$ lies in a sufficiently thin `tubular neighborhood' of $\alpha$ so that each normal line
intersects it in exactly one point. We may then choose the perturbation as $\alpha_\ep(\theta)= \alpha(\theta) + \epsilon \dot{\rho}(\alpha(\theta)) \nu (\alpha(\theta))$. The two constraints $d P(X) = 0 = d L(X)$ can be impressed in terms of $\dot{\rho}$.

Since conditions (1), (2) and (3) imposed on the class $\mathcal D_L$ used in our main theorem are evidently open dense conditions, we only focus on condition (4) which we recall below.
\begin{enumerate}
	\item [(4)] The lengths $2L, 4L$ of $\gamma, \gamma^2$, have  multiplicity one in the length spectrum Lsp$(\Omega)$ and $4L \neq P$ where $P = | \d \Omega |$.  
\end{enumerate}

We show that:
\begin{prop} \label{open dense and generic}Suppose  $4L \neq P$.  Let $\mathcal N_{\sigma, L, P}$ be the class of domains $\Omega$ whose boundaries are parameterized by elements of $\mathcal E^\omega_{\sigma, L, P}$ and such that $2L$ and $4L$ are simple in the length spectrum of $\Omega$. Then $\mathcal N_{\sigma, L, P}$ is residual in $ \mathcal E^\omega_{\sigma, L, P}$. Furthermore, if we let $\mathcal C_{\sigma, L, P}$ to be the class of strictly convex boundaries in $\mathcal N_{\sigma, L, P}$, then $\mathcal C_{\sigma, L, P}$ is open dense in $\mathcal E^\omega_{\sigma, L, P} \cap\{\text{strictly convex boundaries} \} $.  
\end{prop}

\subsection{\label{LENGTHSECT1} Background on transversally reflecting periodic orbits and length functions}

Given a simply connected smooth domain $\Omega$, we denote by $\Gamma(p,q)(\Omega)$ the set of periodic orbits of winding number $p$ and bounce number $q$.  We also let $\Gamma_q (\Omega) = \bigcup_{p \in {\mathbb N}} \Gamma(p,q)(\Omega).$  If all reflections from
the boundary are transversal, then a periodic orbit is  a Snell polygon with $q$ vertices, i.e. a polygon satisfying Snell's law
at every vertex. The possible $q$-bounce periodic transversal reflecting rays, or $q$-vertex Snell polygons, are the critical points of the length function on the configuration space of $q$ points on $\partial \Omega$. 
 When the domain is convex, the length spectrum $Lsp(\Omega)$ is the closure of the lengths of the $(p,q)$ orbits; the only points 
of accumulation are multiples of the perimeter.

When the  domain is non-convex, there can also exist billiard trajectories which glide along convex parts of the boundary. We  first  prove Proposition \ref{open dense and generic} in the convex case.  In Section  \ref{GLIDINGSECT},
we discuss periodic orbits and length functions for  general, possibly non-convex domains. 

We now define
the relevant objects in the case of convex domains.

The configuration space of $q$ points on the parametrizing circle is the $q$-th cartesian product of the circle, i.e. $(\R /\Z)^q $. Under the parametrization $\alpha $ of $\partial \Omega$, $(\R /\Z)^q $ maps
to the configuration space of $q$ points on the boundary of the image domain. We may pull back the length functional and define it on $(\R /\Z)^q $ by,
\begin{equation}\label{LFUN}  \lcal^{(q)}_{\alpha} (\theta_1, \dots, \theta_q) = \sum_{j = 1}^{q} \|\alpha(\theta_{j+1}) - \alpha(\theta_j) \|.
 \end{equation}

The length functions are singular on the singular locus of the configuration space, namely the  coincidence set $\Delta_q:  = \{\vec \theta \in (\R /\Z)^q : \exists j:  \theta_{j+1} = \theta_j \}. $ Points of the coincidence set correspond to $q$ vertex polygons where at least two vertices coincide. These points therefore belong to $(\R /\Z)^{q-1} $ or some lower configuration space. 
We therefore puncture $\Delta_q$ from the Cartesian product
 and define the $q$-fold configuration space by 
 $$\Omega_q: = \left ( \R / \Z \right )^q \backslash \Delta_q. $$
 Note that creeping rays with $q$ links correspond to critical points in $(\R /\Z)^q $ which lie in a $O(\frac{1}{q})$ neighborhood
 of $\Delta_q$.  It is such almost-coincident critical points rather than singular points of $\Delta_q$ that most complicates the analysis. We denote
 by  $T_{\frac{1}{q}} (\Delta_q)$ the open  tubular neighborhood of radius $\frac{1}{q}$ around
$ \Delta_q$. Then  
\begin{equation} \label{COMPACTCONFIG} \wt \Omega_q: = (\R /\Z)^q \backslash T_{\frac{1}{q}} (\Delta_q) \end{equation}
is a compact set on which $\lcal_\alpha^{(q)}$ is real analytic if $\alpha$ is analytic. 
  We define the critical point set of $\lcal_\alpha^{(q)}$  by,
 $$\text{Crit} \; \lcal_{\alpha}^{(q)} = \{ \vec \theta \in \Omega_q: \nabla_{\vec \theta} \lcal_{\alpha} (\vec \theta) = 0\}. $$
 It is well-known that the critical point set maps to vertices of a Snell polygon with $ q$ vertices.
 We also denote the set of critical values by  
 $$ CV(\lcal^{(q)}_{\alpha}) = \{  \lcal_{\alpha} (\vec \theta): \nabla_{\vec \theta} \lcal_{\alpha} (\vec \theta) = 0\}. $$
 The set of critical values of $\lcal^{(q)}_{\alpha}$  corresponds to the set of lengths of $q$ bounce periodic orbits of the domain
 defined by $\alpha$. The number  of critical values of a real analytic function on a compact set  is always finite  $\alpha$ \cite{SS72}. 
 Hence  $$ \# CV \lcal^{(q)}_{\alpha} |_{\wt \Omega_q}  < \infty. $$

 \begin{rema}
 	We note that the zero set of a real analytic function of $q$ variables can have dimension at most $q-1$ and is a stratified analytic
 subvariety. In general the zeros of a real analytic vector field has dimension at most $q-1$ but in the special case of $\nabla \lcal^{(q)}_{\alpha}$ it is at most 1. If the components of the vector field $\nabla \lcal^{(q)}_{\alpha}$
 are independent then the zeros are isolated and of finite order. When a critical point $\vec\theta$ of $\lcal^{(q)}_{\alpha}$ is not isolated, then it corresponds to a degenerate periodic orbit. Examples where this occurs include the disk and ellipses.
 In this case, except for the two axes of the ellipse, periodic orbits of bounce number $q$ come
 in a one-parameter family.
\end{rema}

\subsection{\label{CONVEX} Strictly convex centrally symmetric analytic domains }
In this and the next sections we prove Prop \ref{open dense and generic}. We first consider the case of strictly convex domains which have a prescribed upper and lower bounds on their curvature. Let $0< c_1 <c_2$ and $\mathcal E^\omega_{\sigma, L, P, c_1, c_2}$ be the class of strictly convex boundaries in $\mathcal E^\omega_{\sigma, L, P}$ whose curvature functions take values in the open interval $(c_1, c_2)$.  We then have the following preliminary lemma:
 \begin{lemm} \label{MAINLEM} Suppose $4L \neq P$. Then there exists $q_0$ that depends only on $L$, $P$, $c_1$, and $c_2$, such that for each domain $\Omega$ whose boundary $\d \Omega$ is in $\mathcal E^\omega_{\sigma, L, P, c_1, c_2}$, the number of reflections $q$ of periodic orbits $\gamma$ of $\Omega$ of lengths $L(\gamma) =2L$ or $4L$ is bounded by $q_0$. 
\end{lemm}

\begin{proof} It is more convenient to use the arclength variable $ s\in  \R/ P \Z$ rather than $ \theta \in \R / \Z$. Given $\alpha \in \mathcal E^\omega_{\sigma, L, P, c_1, c_2}$ we use $\wt \alpha: \R/ P \Z \to \R^2$ for its arclength reparametrization in the counterclockwise direction. Suppose $\gamma$ is a $(p, q)$ periodic orbit of $\Omega$ whose boundary is given by $\wt \alpha$. Then $\gamma$ corresponds to a critical point $(s_1, \dots, s_q)$ of
	$$\lcal^{(q)}_{\wt \alpha} (s_1, \dots, s_q) = \sum_{j = 1}^{q} \|\wt \alpha(s_{j+1}) - \wt \alpha(s_j) \|. $$
	Suppose $L(\gamma) =2L$ or $L(\gamma) =4L$. Then there must exists two consecutive vertices $s$ and $s'$ of the $q$-gon $\gamma$ such that 
	\begin{equation} \label {small link} \|\wt \alpha(s) - \wt \alpha(s') \| \leq \frac{4L}{q}. \end{equation}
	Let $\kappa(s)$ be the curvature function of $\tilde \alpha$ at $s$. 
	Since $\| \alpha'(s) \| =1$ and $\| \alpha ''(s) \| = \kappa (s)$, by Taylor's theorem we have
	\begin{equation} \label{Taylor 2nd order}\left \|  \alpha(s) - \alpha(s') - (s-s') \alpha'(s')  \right \| \leq  \frac{\kappa_{\max}}{2}  (s-s')^2, \end{equation}
	where $\kappa_{\max}$ is the maximum of $\kappa$ on $[0, P]$. 
	Our argument divides into two cases:
	
	\textit{Case 1:} Suppose
	\begin{equation} \label{a priori bound on link} | s-s'| < \frac{1}{\kappa_{\max}}. \end{equation}
	Then from \eqref{small link} and \eqref{Taylor 2nd order} we must have
	\begin{equation} \label{small arc} | s-s'| \leq \frac{8L}{q}, \end{equation} where $|s-s'|$ stands for the length of the shortest arc between $s$ and $s'$ in $\R / P \Z$, and not necessarily the the distance in the positive direction. 
	
Let $\phi$ and $\phi'$ be the angle of reflections of the orbit $\gamma$ at $s$ and $s'$, respectively. If we denote $\beta$ be the billiard map of $\Omega$, 
	then they are related by $\beta (s, \phi) = (s', \phi')$. 
	 By Proposition 14.1 of \cite{L}, we know that
	\begin{equation}\label{Lazutkin3}
	\frac{2}{\kappa_{\text{max}}} \phi	\leq |s' -s|  \leq  \frac{2}{\kappa_{\text{min}}} \phi,
	\end{equation}  
	From \eqref{small arc} and \eqref{Lazutkin3}, we get
	$$ \phi < \frac{4L \kappa_{\max}}{q} <  \frac{4L c_2}{q}.$$
	With a bit of work, using Lazutkin coordinates, one can show that there exists a constant $C$ that depends only on $L$, $c_1$ and $c_2$ that all angles of reflections $\{\phi_j\}_{j=1}^q$ of the orbit $\gamma$ are bounded by 
	\begin{equation} \label{small angles} \phi_j < \frac{C}{q}. \end{equation}
	For a proof, see for example Lemma 4.1 of \cite{Vig}. 
	Now,  let $\wt s_1, \dots, \wt s_q$ denote the lifts of $s_1, \dots, s_q$ from $\R/ P\Z$ ro $\R$. 
	By the definition of the winding number $p$ of $\gamma$, we have 
	$$\sum_{j=1}^q (\wt s_{j+1} - \wt s_j) = p P. $$ 
	Then by \eqref{Lazutkin3} and \eqref{small angles}, we get 
	$$ pP  \leq \frac{2}{ \kappa_{\min}} \sum_{j=1}^q \phi_j <  \frac{2C}{ \kappa_{\min}}, $$
	which implies $$ p < \frac{2C}{ c_1P}. $$
	We note that this estimate holds for any $(p, q)$ periodic orbit $\gamma$ of the entire class $\mathcal E^\omega_{\sigma, L, P, c_1, c_2}$ as long as there is a link for which \eqref{a priori bound on link} holds and $L(\gamma) \leq 4L$. We now estimate the length of $\gamma$ in terms of the perimeter of $\wt \alpha$. First we observe that by \eqref{Lazutkin3}, we get
	$$ | \wt s_{j+1}- \wt s_j | < \frac{2}{c_1}{ \phi_j} < \frac{2C}{c_1q}. $$ Now we can use \eqref{Taylor 2nd order} and write
	\begin{align*}   | L (\gamma) - p P |  &=  \left | \, \sum_{j=1}^q \| \wt{\alpha}(s_{j+1}) - \wt{\alpha}(s_j)  \| - \sum_{j=1}^q | \wt s_{j+1} - \wt s_j |  \, \right | \\
	&  \leq  \frac{c_2}{2}  \sum_{j=1}^q | \wt s_{j+1} - \wt s_j |^2 \\
	& \leq \frac{2c_{2} C^2}{c_{1}^2 q}.  
	\end{align*}

                        In other words, 
                        $$ | L(\gamma) - p P | \leq  \frac{K}{q},$$
	where the constant $K$ depends only on $L$, $c_1$ and $c_2$. Now by our assumption that $L(\gamma) =2L$ or $4L$, we get 
	$$ | 2L - pP|  \leq  \frac{K}{q}  \qquad \text{or}  \qquad | 4L - pP| \leq \frac{K}{q}.$$
	Since $2L <P$ and $4L \neq P$, we obtain a contradiction if $q$ is too large in terms of $L$, $P$, $c_1$, and $c_2$. 
	
	It now remains to deal with:
	
	\textit{Case 2: } We have a small link but with a large arc, i.e. 
	$$  \|\wt \alpha(s) - \wt \alpha(s') \| \leq \frac{4L}{q} \qquad \text{but} \quad  |s-s'| \geq \frac{1}{\kappa_{\max}}.$$
	The next lemma shows that for sufficiently large $q$ this case does not happen when one has upper and lower bounds $c_1$ and $c_2$ for the curvature. This would end the proof of Lemma \ref{MAINLEM}. 
\end{proof}

	\begin{lemm} There exists $q_0$ such that there is no $\alpha \in \mathcal E^\omega_{L, P, c_1, c_2}$ with a $(p, q)$-type periodic orbit $\gamma$, $q \geq q_0$,  that has a link $ \overline{\wt \alpha(s)  \wt \alpha(s')}$ satisfying 
	$$  \|\wt \alpha(s) - \wt \alpha(s') \| \leq \frac{4L}{q} \qquad \text{but} \quad  |s-s'| \geq \frac{1}{\kappa_{\max}}. $$
\end{lemm}
\begin{proof} We align the Cartesian coordinates so that the segment $ \overline{\wt \alpha(s)  \wt \alpha(s')}$ is on the $x$-axis. Let $Q_1= \wt \alpha(s) $ and $Q_2 = \wt \alpha(s')$ be the right and left intersection points of $\wt \alpha$ with the $x$-axis, respectively. If the tangent lines to $\wt \alpha$ at $Q_1$ and $Q_2$ intersect at a point $T$, we reflect $\wt \alpha$ about the $x$-axis, if necessary,  to place $T$ on the upper half plane. If the tangent lines at $Q_1$ and $Q_2$ are parallel we do not make any reflections. We note that with this setting the upper arc that connects $\wt \alpha(s)$ to  $\alpha(s')$ lies completely inside the triangle $Q_1TQ_2$. Next, we note that there must exists a point $s_0$ on the upper arc whose tangent line is parallel to the $x$ axis and we let $y =t_0$, $t_0>0$, be this tangent line. We then let $f(t)$ to be the length of the part of the line $y=t$ that intersects the domain. Then by the definition of $t_0$ we have $f(t_0) =0$, and by the the convexity of the domain, $f(t)$ is strictly decreasing on $ [0, t_0]$ (see for example Lemma 1.1 of \cite{Am}).  We then locate a disk $S$ of radius $\frac{1}{\kappa_{\max}}$  that is tangent to $\alpha$ at $s_0$ from below. By a theorem of Blaschke (\cite{Bl}, page 116), the entire disk $S$ lies inside $\wt \alpha$. Let $D$ be the diameter of $S$ that is parallel to the $x$-axis and let $S_1$ be the open half-disk that lies above $D$.   If $S_1$ does not intersect the $x$-axis then we claim that $|D|$ must be shorter than the segment $ \overline{\wt \alpha(s)  \wt \alpha(s')}$.  This is because under this assumption, the diameter $D$ lies on $y =t_1$ for some $t_1>0$, and therefore by the monotonicity of $f$,
	$$|D| \leq  f(t_1) <  f(0) =  |\overline{\wt \alpha(s)  \wt \alpha(s')} |. $$
This would imply that 
$$ \frac{2}{\kappa_{\max}} < \frac{4L}{q},$$
which leads to a contradiction if $q \geq 2L \kappa_{\max}$. 
Now suppose the half-disk $S_1$ intersects the $x$-axis and let $\ell$ be the cord of $S_1$ that lies on the $x$-axis. Because the upper part of the domain, hence also $S_1 \cap \{ y >0 \}$, lie inside the triangle $Q_1TQ_2$, the half-disk $S_1$ must in fact  intersect the $x$-axis on the segment  $|\overline{\wt \alpha(s)  \wt \alpha(s')} |$. As a result,
$$|\ell| \leq  |\overline{\wt \alpha(s)  \wt \alpha(s')} | \leq \frac{4L}{q}.$$ 
Since the distance from $s_0$ to $\ell$ (which equals $t_0$) is less than the length of the cord $\ell$, we obtain 
$$t_0 < \frac{4L}{q}. $$
 On the other hand the upper arc of $\wt \alpha$ is contained between the lines $y=0$ and $y = t_0$. Thus the upper part of the domain is contained in the thin strip $ 0 \leq y \leq \frac{4L}{q} $. Let $V$ and $W$ be the two points on the boundary that have vertical tangent lines (there are exactly two by strict convexity). We assume $x(V) < x(W)$, i.e. $V$ is on the left and $W$ is on the right. Since by our assumption, the tangent lines at $Q_1$ and $Q_2$ intersect in the upper half-plane (or they are parallel), it is clear that $x(W) \geq x(Q_2)$ or $x(V) \leq x(Q_1)$. Without the loss of generality we assume $x(W) \geq x(Q_2)$. If in addition also $x(V) \leq x(Q_1)$, then the upper arc of the domain must be contained in the rectangle $ \{ x(Q_1) \leq x \leq x(Q_2),  0 \leq y \leq \frac{4L}{q} \} $, which by convexity implies that, $|s-s'|$, the length of the upper arc is less than 
$$2 \times \frac{4L}{q} + x(Q_2) -x(Q_1) \leq \frac{12L}{q}.$$
Since by assumption $|s-s'| \geq \frac{1}{\kappa_{\max}}$, we obtain a contradiction for $q$ sufficiently large.  Finally assume $x(V) > x(Q_1)$. Then $V$ must be in the upper half-plane because the entire domain is contained in the sector $\widehat{Q_1TQ_2}$. Hence $V$ must belong to the strip $ 0 \leq y \leq \frac{4L}{q}$. We then consider the disk of radius $\frac{1}{\kappa_{\max}}$ that is tangent to $\wt \alpha$ at $V$ from the right. The entire disk must be contained inside the domain. However, the center of the of disk is also contained in $ 0 \leq y \leq \frac{4L}{q}$, thus if we choose $q$ large enough so that $\frac{8L}{q} < \frac{1}{\kappa_{\max}}$, then there must be a point on the disk that lies above the line $y = \frac{1}{\kappa_{\max}} -  \frac{4L}{q} $, which is a contradiction with the assumption that the domain is under the horizontal line $y = t_0 < \frac{4L}{q}$. 
	\end{proof}

The following lemma asserts that when $q$, the number of bounces, is fixed then it is an open dense condition that the lengths of $q$-periodic orbits are simple. 
\begin{lemm} \label{class O}  

Suppose $4L  \neq P$.  Let $q  \in \mathbb N$ and $O_{L, P, c_1, c_2, q}$ be the class of
	$\alpha \in \mathcal E^\omega_{\sigma, L, P, c_1, c_2}$ such that critical points of $ \lcal^{(q)}_{\alpha}$ are all non-degenerate and the critical values are distinct up to the action of the cyclic group $\Z_q$ on $(\R / \Z)^q$.
	Then $O_{L, P, c_1, c_2, q}$ is open dense in $\mathcal E^\omega_{\sigma, L, P, c_1, c_2}$. 
\end{lemm}

\begin{proof} The statement is well-known for the space of $C^{\infty}$ embeddings. By an argument given in the Appendix Section \ref{APPENDIX},
one can deduce the statement in the space of real analytic embeddings by controlled approximation of smooth embeddings by real analytic ones.
We therefore only sketch the argument.

Since $P$ is fixed,  we shall use the arclength parameterization to denote the configuration space, 
	\begin{equation*} \label{NCONFIG} \Omega_q =(\R / P \Z)^q \backslash \Delta_q.  \end{equation*}
	Moreover, by the proof of the previous lemma, the links cannot be shorter than $C/{q}$ so we may delete an open tubular  neighborhood of radius  $\frac{C}{q}$ around
	$ \Delta_q$ to obtain  a compact set 
	\begin{equation} \label{COMPACTCONFIGb} \wt \Omega_q= (\R / P \Z)^q \backslash T_{\frac{C}{q}} (\Delta_q).  \end{equation} The constant $C$ only depends on $L$, $P$, $c_1$ and $c_2$.  The length function is real analytic on $\wt \Omega_q$. As a result, for any $\alpha$, it has only
	finitely many critical values. We then use an abstract argument showing that the class $O_{L, P, c_1, c_2, q}$, defined in the statement of the lemma, is open and dense.  
	
	The function $\lcal_{\wt \alpha}^{(q)}(s)$ has only finitely many critical values on $\wt \Omega_q$ \cite{SS72}. If $\alpha_0$ is such that the critical points of $ \{\lcal_{\wt \alpha_0}^{(q)} \} $
	are non-degenerate with distinct critical values, then the same is true for any analytic perturbation of $\lcal_{ \wt \alpha_0}^{(q)}$ which is small in the
	real analytic topology.  Thus openness follows. 
	
	To prove denseness we first prove denseness of the non-degeneracy condition. 	
 The condition that the critical point is non-degenerate is equivalent to the statement that $\nabla \lcal_{\wt \alpha}^{(q)}$ is transversal to the
	zero section of $T^* \wt \Omega_q$. By the Thom transversality theorem, the non-degeneracy is an open dense property of smooth functions; 
	using an approximation argument as in the Appendix Section \ref{APPENDIX}, it is also an open dense property for analytic functions (see also
	\cite[Theorem 3.3]{BeMa}).
	

	Next, we assume all critical points of $\lcal_{ \wt\alpha}^{(q)} $  are non-degenerate, hence in particular isolated, and prove the denseness of the distinctness property of the critical values.   So suppose that there exists an open neighborhood $\ucal$ around $\alpha_0$ so
	that for all $\alpha \in \ucal$, there exist $x_\alpha, y_\alpha \in \wt \Omega_q$ that are distinct up to the action of cyclic group $\Z_q$, such that 
	$$ \nabla  \lcal_{ \wt \alpha}^{(q)}(x_\alpha)= 0 = \nabla  \lcal_{ \wt \alpha}^{(q)}(y_\alpha),$$
	and  $$  \lcal_{ \wt \alpha}^{(q)}(x_\alpha) = \lcal_{ \wt \alpha}^{(q)}(y_\alpha). $$
By the non-degeneracy of critical points, we know that $x_{\alpha}$ and $y_{\alpha}$ vary smoothly in $\alpha$. Taking first variations, we obtain
$$d  \left ( \nabla  \lcal_{ \wt \alpha}^{(q)}(x_\alpha) \right )\big\rvert _{\alpha_0} (X_0)=0=d  \left ( \nabla  \lcal_{ \wt \alpha}^{(q)}(y_\alpha) \right )\big\rvert _{\alpha_0} (X_0), $$ 
and $$  d \left (\lcal_{ \wt \alpha}^{(q)}(x_\alpha) \right )\big\rvert _{\alpha_0} (X_0) =  d \left ( \lcal_{ \wt \alpha}^{(q)}(y_\alpha) \right) \big\rvert _{\alpha_0}(X_0).$$
Again, because the infinitesimal space $T_{\alpha_0} \mathcal E^\omega_{\sigma, L, P, c_1, c_2}$ is rich enough, there exists $X_0$ that separates these equations (including the constraints \eqref{P and L preserved}). 

\end{proof}
We are now in position to prove the main result of this section in the strictly convex case.  

\begin{proof}[\textbf{Proof of Proposition \ref{open dense and generic}, convex case}]
Suppose  $4L \neq P$.  Let $\mathcal C_{\sigma, L, P}$  be the class of strictly convex domains $\Omega$ whose boundaries are parameterized by elements of $\mathcal E^\omega_{\sigma, L, P}$ and such that $2L$ and $4L$ are simple in the length spectrum of $\Omega$.  

We first note that
$$\mathcal E^\omega_{\sigma, L, P, 0, \infty} = \bigcup_{n=1}^\infty \mathcal E^\omega_{\sigma, L, P, \frac{1}{n}, n} .$$
This is precisely the class of strictly convex domains in $\mathcal E^\omega_{\sigma, L, P}$  with no further assumptions (hence no simplicity assumptions on lengths). Let $q_0(n)$ be the bounce number found in Lemma \ref{MAINLEM} for the class $\mathcal E^\omega_{\sigma, L, P, \frac{1}{n}, n}$ so that all $q$-periodic orbits, $q > q_0(n)$, have lengths $> 4L$ . By Lemma \ref{class O}, 
$$ I_n:= \bigcap_{q=1}^{q_0(n)} O_{L, P, \frac{1}{n}, n, q} $$
is open dense in  $\mathcal E^\omega_{\sigma, L, P, \frac{1}{n}, n}$ and by Lemma \ref{MAINLEM} we have $ I_n \subset \mathcal C_{\sigma, L, P}$.  We claim that 
$$ \bigcup_{n=1}^\infty I_n  = \bigcup_{n=1}^\infty \bigcap_{q=1}^{q_0(n)} O_{L, P, \frac{1}{n}, n, q}  $$
is open dense in $\mathcal E^\omega_{\sigma, L, P, 0, \infty} = \bigcup_{n=1}^\infty \mathcal E^\omega_{\sigma, L, P, \frac{1}{n}, n}$. The denseness is obvious. For the openness, we note that $\mathcal E^\omega_{\sigma, L, P, \frac{1}{n}, n}$ is open in $\mathcal E_{\sigma, L, P, 0, \infty}$, thus $I_n$ is open in $\mathcal E^\omega_{\sigma, L, P, 0, \infty}$ and hence also the countable union $\bigcup_{n=1}^ \infty I_n$. Finally, since 
$$  \bigcup_{n=1}^ \infty I_n \subset \mathcal C_{\sigma, L, P}, $$ 
 $\mathcal C_{\sigma, L, P}$ is also dense. The openness follows from a similar argument as in the proof of Prop \ref{class O}; we only need to change the assumption of distinctness of all critical values, to distinctness from $2L$ and $4L$. 

\end{proof}

\subsection{\label{GLIDINGSECT} Proof of Proposition \ref{open dense and generic}, non-convex case} 

To prove Proposition \ref{open dense and generic} for non-convex domains, we begin by extending the discussion of periodic orbits
and length functions from transversal reflecting rays of Section \ref{LENGTHSECT1} to the more general types of periodic orbits which can exist
for non-convex domains. We assume throughout the discussion that the domain is real analytic.

   A real analytic boundary consists of a finite number
of convex parts, a finite number of concave parts, separated by (possibly degenerate)  inflection points at which curvature vanishes and the boundary changes concavity.
Non-convex domains  may possess periodic orbits that are not $(p,q)$ periodic transversally reflecting rays. Namely, there may exist  periodic
orbits which glide along some of the (disjoint)  convex parts of the boundary.  Examples of periodic orbits which glide along one or more convex parts of the boundary are illustrated in \cite{GuMe79} and \cite{PS17}.  In the non-convex case, there do not exist periodic orbits which glide along the
entire boundary; they can only glide along the convex parts. 

We call a periodic orbit   ``gliding/linear'' if it has non-empty gliding part. A crucial feature of such orbits is that the gliding segments must coincide with a full convex part,
hence the entry and exit points of the orbit must be inflection points and there must exist linear segments at each endpoint which are tangent to
the boundary at the inflection point. With no loss of generality, we may assume the curvature vanishes to order 1 at each inflection point, and then
the linear segment has a unique continuation to the full convex part into which it enters (or, in reverse time, exits; see Section \ref{INFLECT}). Consequently, to find critical 
configurations of points on $\partial \Omega = \alpha(\R/\Z)$ corresponding to gliding/linear periodic orbits, it is only necessary to ensure that
the linear segments (`links') that touch inflection points do so tangentially to $\partial \Omega$.
 Note that it is possible for  a linear segment to
enter  the boundary tangentially at both of its  endpoints, so that the orbit glides along two `opposite' convex segments, or to   hit the boundary transversally at the second  endpoint.

  
  

According to \cite[Theorem 9.4]{PS87}, generic $C^{\infty}$ simply connected domains in $\R^2$  have no periodic billiard trajectories containing
both gliding segments and linear segments. The new step in the non-convex case is to prove (or, cite) the analogous result in the analytic category.

\begin{lem} \label{ANALYTICNC} A generic real analytic non-convex simply connected domain has no gliding/linear periodic billiard trajectories.
\end{lem}

\subsubsection{\label{INFLECT} Inflection points}

The first technicality is that the exit and entry points should be inflection points of the boundary, where the curvature vanishes to order 1.
It is proved in  \cite[Theorem 9.4]{PS87} that the set of embeddings for which all points of vanishing curvature are inflection points is
residual. The main step is to prove that the subset of  3-jets of embeddings for which the curvature vanishes to second or higher
order at some point is a submanifold of codimension $2$. The proof applies with almost no change to real analytic domains, and we only sketch
the proof. The  condition that the curvature
 vanish to order $\geq 2$ at a point $x$  is given by explicit equations \cite[(B.1)-(B.2)]{PS87}. Since the 3-jet space is the same for $C^{\infty} $ and $C^{\omega}$ domains, the same proof implies that
the set of 3-jets of $C^{\infty} $ domains containing a point where the curvature vanishes to   order $\geq 2$ at some point is of codimension $\geq 2$.

In the case of a real analytic domain, there can only exist a finite number of points where the curvature vanishes.  It suffices to show that for each
number $m$ of points where the curvature vanishes it is an open dense condition that the curvature vanishes only to order $1$. But this follows from
the fact that the set of embeddings where the curvature vanishes to order $\geq 2$ is closed and nowhere dense.

\subsection{Proof of Lemma \ref{ANALYTICNC} }
Let $\gamma$ be a generalized geodesic ray and let    $\{x_1, \dots, x_s\}$ denote the endpoints of the linear segments. By `generalized geodesic ray' it is understood that consecutive transversal links satisfy the reflection law of equal angles at their common endpoints.  There
exists an oriented ordering of the endpoints by $\iota: \{1, \dots, k\} \to \{1, \dots, s\}$, where $k$ is the number of distinct elements of 
$ \{1, \dots, s\}$.\footnote{$\iota$ is denoted by $\omega$ in \cite{PS87}.}. so that the successive  linear segments with respect to
the counter-clockwise orientation of $\alpha(S^1)$ are given by $\ell_j = [\alpha(x_{\iota(j)}), \alpha(x_{\iota(j+1)})]$. 
Following \cite[Section 2.6]{PS87} (see also \cite[Pages 660-661]{PS87}), we denote by $I_i(\omega) $ the set of indices $j$ for which there exists a segment of $\gamma$ joining
$x_i$ to $x_j$ and we denote by $U_{\iota} = \{(x_1, \dots, x_s) \in (\R^2)^s: x_i \notin CH \{x_j: j \in I_i(\iota)\}, \forall i \in \rm{Im} (\iota)\}$ (where
CH denotes the convex hull).
Periodic reflecting rays and gliding/linear rays of type $\iota$ always lie in $U_{\iota}$. 

We denote by $T^{\infty}_{\iota}$, resp. $T^{\omega}_{\iota}$,  the set of smooth, reps. analytic, embeddings $\alpha$ into $\R^2$ such that for every $\vec x \in (\partial \Omega)^{(s)}$ with $\alpha^s(x) = (\alpha(x_1), \dots, \alpha(x_s)) \in U_{\iota}$, 
$$\kappa_{\alpha} (x_1) = \kappa_{\alpha}(x_s) = 0 \implies \left\{\begin{array}{ll} \langle \overline{\alpha(x_2) - \alpha(x_1)}, \nu_{\alpha(x_1)} \rangle
\not=0, \; & \text{or} \\ &\\
\nabla_{x'} \lcal_{\alpha}(x_1, x', x_{s}) \not=0. & \end{array} \right. $$
In other words, if the intial link $\overline{\alpha(x_2) - \alpha(x_1)}$ hits $\partial \Omega$ at an inflection point $x_1$ and the terminal 
link hits an inflection point,  then either the initial link  $\overline{\alpha(x_2) - \alpha(x_1)}$  is transversal to
$\partial \Omega$ at $\alpha(x_1)$ or the polygonal trajectory does not satisfy the Snell law of equal angles at the intermediate points, i.e. is
not a generalized geodesic. It is proved in \cite{PS87} that $T^{\infty}_{\iota}$ is residual in $C^{\infty}_{\rm{emb}}(S^1, \R^2)$. 

Let $J_s^2(S^1, \R^2)$ denote the space of $2$-jets of smooth embeddings $$(\alpha_1(x_1), \dots, \alpha_s(x_s)): (S^1)^s \to \R^{2s}.$$
Define the source, resp. target, maps by,
$$\beta_1(j^2 f(x)) = x, \; \text{resp.}\;\; \beta_2(j^2 \alpha(x)) = \alpha(x), $$
and extend them to maps $\beta_1^s$, resp. $\beta_2^s$, on  symmetric products (cf. \cite[Page 627]{PS87}).
  Given an open subset
$U \subset (\R^2)^s$ let $M = (\beta_1^s)^{-1} (S^1)^{(s)} \cap (\beta_2^s)^{-1}(U). $ Let
\begin{equation} \label{sigma} \sigma = (j^2 \alpha_1(x_1), \dots, j^2 \alpha_s(x_s)) \in M.  \end{equation}
Let $\Sigma $ be the set of $\sigma$ in \eqref{sigma} such that (see \cite[Page 661]{PS87}):

\begin{itemize}
\item  the curvature of $\alpha_1(S^1)$, resp. $\alpha_s(S^1)$, vanishes at  $\alpha_1(x_1)$, 
resp. $\alpha_s(x_s)$; 
\item $\nabla_{x'} \lcal^{k}_{ (\alpha, \dots, \alpha_s), \iota} (x_1, \dots, x_s)  =0$, where $$\lcal^{k}_{ (\alpha, \dots, \alpha_s), \iota} (x_1, \dots, x_s) = \sum_{j = 1}^{k} \|\alpha_{\iota(j+1)} (x_{\iota(j+1)}) -\alpha_{\iota(j)}(x_{\iota(j)} ) \|;$$  
\item $ \langle \overline{\alpha_2 (x_2) - \alpha_1(x_1)}, \nu_{\alpha_1(x_1)} \rangle =0$. 
\end{itemize}

 Given $\alpha \in C^{\infty}(S^1, \R^2)$ define the induced embedding,

$$\left\{\begin{array}{ll}j_s^2 \alpha^s: (S^1)^{(s)} \rightarrow J_s^2 (S^1, \R^2), \; &  \\ &\\ j_s^2 \alpha^s(x_1, \dots, x_s) := (j^2 \alpha(x_1), \dots, j^2 \alpha(x_s)). 
	 & \end{array} \right.$$

 By definition, \begin{equation} \label{Tomega} T^{\infty}_{\iota} = \{\alpha \in C^{\infty}_{\rm{emb}}(S^1, \R^2): j_s^2 \alpha^{s} (S^1)^{(s)} \cap \Sigma = \emptyset\}. \end{equation}
 
According to \cite[Lemma 9.2]{PS87}, \begin{equation} \label{CODIMSIGMA} \codim \Sigma = s +1.  \end{equation}
This calculation is the same for real analytic and for smooth embeddings. It follows from \eqref{Tomega} and from \eqref{CODIMSIGMA} that,
\begin{equation} \label{Tomega2} T_{\iota}^{\omega} = \{\alpha \in C^{\omega}_{\text{emb}}(S^1, \R^2): j_s^2 \alpha^{s} \transv \Sigma\}. \end{equation}
(Indeed, for the above dimensions, transversality is equivalent to non-intersection. E.g. for 0-jet transversality, if $f: M \to N$ and $Q \subset N$,
if $\dim M < \rm{codim} Q$,
then $f \transv Q$ if and only if $f(M) \cap Q = \emptyset$.)

The equation \eqref{Tomega2} is a ``Kupka-Smale'' reformulation of the desired transversality condition in the sense of \cite{BrT86} (see
Section \ref{APPENDIX}). The multi-jet transversality theorem (see Section \ref{APPENDIX}) in the $C^{\infty} $ category implies that \eqref{Tomega} is residual in $C^{\infty}_{\rm{emb}}(S^1, \R^2).$
Since   $\iota$ varies  in a countable set, the set of domains without  periodic gliding/linear rays is residual. By the argument of \cite{BrT86} (see
Section \ref{APPENDIX}), the sets \eqref{Tomega2} are residual in  $C^{\infty}_{\rm{emb}}(S^1, \R^2)$.

For the sake of expository clarity, we sketch a proof here. Let $\alpha_0$ be a real analytic embedding. We wish to embed $\alpha_0$ into a sufficiently
rich family $F$ of real analytic maps so that the multi-2-jet $j_s^2 F$ of the family is transversal to $\Sigma$. By the multi-jet transversality theorem,
this is possible in the $C^{\infty}$ category. Then we approximate $F$ by a real-analytic family $F^{\omega}$ using the heat kernel technique
of \cite{BrT86}. For small enough time in the heat kernel, the family remains transversal to $\Sigma$. It then follows that the set of elements in
the family which are not transversal to $\Sigma$ has positive codimension in the family. This proves existence of small real analytic perturbations
which are transversal to $\Sigma$ (hence dense-ness). Moreover, the transversality theorem also proves that the set of $C^{\infty}$ transversal
maps is open for any $s$ and $\iota$ and cutoff from the diagonals of the configuration spaces,  and therefore it is open in $C^{\omega}$. Since we need to take the intersections over $s, \iota$ and cutoffs, the set of real analytic $\alpha$ which have no gliding/linear periodic rays is residual
in $C^{\infty}$.

 This proves Proposition \ref{open dense and generic} in the non-convex case.

\section{\label{APPENDIX} Appendix on the multi-jet transversality theorem and Kupka-Smale conditions in the $C^{\omega}$ category}

In this appendix, we review some definitions and results concerning genericity of various properties in  real analytic function spaces.
We first review the Thom transversality theorem and the Thom jet transversality theorem for $f \in C^{\infty}(M, N)$, the space of smooth maps $f: M \to N$ between two compact manifolds
(see \cite[Section 2, Theorem 2.1]{Hir} for the Thom transversality theorem and  \cite[Section 3, Theorem 2.8]{Hir} for the jet transversality theorem).
The simple transversality theorem is, 

\begin{theo} Let $M, N$ be two $C^{\infty}$ manifolds and let $A \subset N$ be a smooth submanifold. Let $$T(A): = \{f \in C^{\infty}(M, N):
 f \transv A\}. $$
Then,
\begin{itemize} 
\item For any submanifold $A$, $T(A) $ is residual in $C^{\infty}(M, N)$ (hence, dense); \bigskip 
\item If $A$ is compact and without boundary,  then $T(A) $ is open-dense in $C^{\infty}(M, N)$.

\end{itemize}

\end{theo}

The jet transversality theorem is,

\begin{theo} Let $M, N$ be two $C^{\infty}$ manifolds and let $W \subset J^k(M, N)$ be a smooth submanifold and for any $k \geq 0$,  let $$T(W): = \{f \in C^{\infty}(M, N):
j^k f \transv W\}. $$
Then,
\begin{itemize} 
\item For any submanifold $W$, $T(W) $ is residual in $C^{\infty}(M, N)$ (hence, dense); \bigskip 
\item If $W$ is compact and without boundary  in $J^k(M, N)$, then $T(W) $ is open in $C^{\infty}(M, N)$.

\end{itemize}

\end{theo}

\begin{rema} \label{SARD} 
The transversality theorems are often proved by means of the Morse-Sard theorem, which says that if $f: M \to N$ is a smooth map then 
$f(\Sigma_f)$ has measure zero in $N$, where $\Sigma_f$ is the set of critical points of $f$ (see e.g. \cite[Section 3, Theorem 1.3]{Hir}). 
In particular, the set of regular values is residual. 

For real analytic maps, the Sard theorem can be replaced by the stronger result that the appropriate set of critical values is a stratified analytic manifold
of positive codimension.  See for instance,  \cite{BeMa}, especially \cite[Theorem 1.4]{BeMa}.
\end{rema}

The article  \cite{BrT86} of Broer-Tangerman gives  a general method for proving that ``Kupka-Smale'' properties of functions or maps
are generic in $C^{\omega}$ spaces if they are generic in $C^{\infty}$ spaces. By a Kupka-Smale property is meant a property that can
be formulated as a transversality condition on a family of maps. To be more precise, the authors consider the example of vector fields $X$
on a compact manifold. Let $M$ and $N$  be compact (finite dimensional) manifolds, and let $S \subset N$ be a stratified manifold.  
 A property ${\mathfrak B}$  of  vector fields  is referred to as a 'Kupka Smale' property if the following holds: Suppose that  a smooth map $ f_X: M \to  N $ is associated to the vector field $X$,  where $f_X$ depends smoothly on X.  Suppose that property $\mathfrak B$ can be reformulated as:  $X$ has property ${\mathfrak B}$ is and only if $f_X \transv S$.   Further suppose that there exists a  smooth $s$-parameter unfolding $X^{\mu}$ of $X^0 = X$,  so that for
$F: M \times \R^s \to N$, defined by $F(\cdot, \mu) = f_{X^{\mu}}$, one has $F \transv_{M\times \{0\}} S$. One can then apply the
Thom transversality principle to obtain that
$$\{\mu \in \R^s: X^{\mu} \; \rm{has\; property\;} \mathfrak B\} = \{\mu \in \R^s: f_{X^{\mu}} \transv S\} $$
is open and dense in a neighborhood of the origin $0 \in \R^s$. It follows that  $X$ can be approximated by $X^{\mu}$ with property $\mathfrak B$, and
therefore the property is dense. Moreover,  nearby $s$-parameter families induce maps which are also transverse to $S$, hence ${\mathfrak B}$ is open. It is proved in 
\cite[Lemma 3.3]{BrT86}, that if ${\mathfrak B}$ as above is open dense in the  $C^k$ space of vector fields then it is also open and dense in the
space of real analytic vector fields.  The proof is to approximate the smooth family by a real analytic family by using the heat equation to make analytic
approximations. In our setting we approximate the Fourier series \eqref{FS1} by a real analytic Fourier series using the heat equation on the circle.

Suppose then  that $X^{\mu}$ has property ${\mathfrak B}$ for an open dense set of $\mu \in \R^s$. Then there
exists a real analytic unfolding $\wt X^{\mu, t}$ with $t > 0$ sufficiently small to ensure that $\wt X^{\mu, t}$ has property $\mathfrak B$ for $\mu$
open-dense in some neighborhood of $0$ in $\R^s$. Now let ${\mathcal U}$ be a neighborhood of $X$ in the real analytic topology. Then there
exists $\mu$ with $|\mu$ sufficiently small to ensure that $\wt X^{\mu, t} \in {\mathcal U}$ has property ${\mathfrak B}$. The openness of ${\mathfrak B}$ follows from the fact that the real analytic topology is stronger than the $C^{\infty}$ -topology.

Thus, there are two steps to prove the genericity of a property ${\mathfrak B}$ in $C^{\omega}$: (i) the property must be formulated 
as the transversality of a map; (ii)  a $C^{\infty}$  unfolding must be found for which the transversality is known to hold. 
  
In the case of  Lemma \ref{class O}, the reformulation is standard: As mentioned in the proof, the condition
	that the critical point is non-degenerate is equivalent to the statement that $\nabla \lcal_{\wt \alpha}^{(q)}$ is transversal to the
	zero section of $T^* \wt \Omega_q$. Indeed, the use of the transversality theorem to prove genericity of Morse functions with distinct values
	at distinct critical points  in $C^{\infty}$ is standard
	\cite{GG73}. 
	In the case of  Lemma \ref{ANALYTICNC}, the reformulation is carried out in \cite[Section 9]{PS87} for exactly the same purpose 
	in the $C^{\infty}$ category.
	
We also recall that the $C^{\infty}$ multi-jet transversality theorem is the following statement: Let $N, P$ be smooth manifolds and denote by  $J_s^r(N, P)$
 the $s$-fold $r$-jet bundle. Given a  smooth map $f: N \to P$ let $j_s^r(f): N^{(s)} \to J_s^r(N, P)$ be defined by,
 $$j_s^rf(x_1, \dots, x_s) = (j^r f(x_1), \dots, j^r f(x_s)). $$
 Let $Q \subset J_s^r(N, P)$ be a smooth submanifold. Then, $\tcal_Q: = \{f \in C^{\infty}(N, P): j_s^r(f) \transv Q \}$ is generic. If $Q$ is compact, then 
 $\tcal_Q$ is open.

\begin{rema}  Another version of the transversality theorem is this: If $W$ is a submanifold of codimension $c$ in $J_s^k(X, Y)$ then
$C^{\infty}(X, Y)$ contains a residual subset such that $(j_s^k(f))^{-1}(W)$ is of codimention at least $c$ in $X^{(s)}$. 

It is stated in \cite[Section 4.3, Remark 3]{Wi04} that ``similar  results hold for the function spaces of type $C^{\omega}$, i.e. real analytic mappings, which in fact can be deduced from the transversality results for $C^{\infty}$ maps, using the fact that $C^{\omega}$ maps are dense in $C^{\infty}$. No  proof nor any reference is given there but in effect this is proved in \cite{BrT86}.

\end{rema}


\begin{thebibliography}{HHHH}
	
\bibitem[Am13]{Am} R. V. Ambartzumian, \emph{Parallel x-ray tomography of convex domains as a search problem in two dimensions}, Izv. Nats. Akad. Nauk Armenii Mat. \textbf{48} (2013), no. 1, 37--52; reprinted in 
J. Contemp. Math. Anal. \textbf{48} (2013), no. 1, 23--34.


\bibitem[BB72]{BB72} Balian, R.; Bloch, C. Distribution of eigenfrequencies for the wave equation in a finite domain. III. Eigenfrequency density oscillations. Ann. Physics 69 (1972), 76-160.

\bibitem[BeMa]{BeMa} P. Bernard and V.  Mandorino, 
Some remarks on Thom's transversality theorem.  Ann. Sc. Norm. Super. Pisa Cl. Sci. (5) 14 (2015), no. 2, 361-386.

\bibitem[BiMi20]{BM20} M. Bialy and A. E. Mironov, The Birkhoff-Poritsky conjecture for centrally symmetric billiard tables, arXiv: 2008.03566.

\bibitem[Bl16]{Bl} W. Blaschke, \emph{Kreis und Kugel}, Viet, Leipzig, 1916; reprint, Chelsea, New York, 1949.

\bibitem[BrT86]{BrT86} H. Broer and F. Tangerman. From a differentiable to a real analytic perturbation theory, applications to the Kupka Smale theorems. Ergodic Theory and Dynamical Systems (1986), 6(3):345-362.

\bibitem[CdV]{CdV} Y. Colin de Verdi\`ere, 
Sur les longueurs des trajectoires périodiques d'un billard. South Rhone seminar on geometry, III (Lyon, 1983), 122-139,
Travaux en Cours, Hermann, Paris, 1984.

\bibitem[Cl20]{Cl20} A. Clarke, Generic properties of geodesic flows on analytic hypersurface of Euclidean space, arXiv: 1908.04662. 

\bibitem[CRL]{CRL} S. C. Creagh, J. M. Robbins, and R. G.  Littlejohn,  Geometrical properties of Maslov indices in the semiclassical trace formula for the density of states. Phys. Rev. A (3) 42 (1990), no. 4, 1907-1922. 

\bibitem[Du88]{Durso} Durso, Catherine, \emph{On the inverse spectral problem for polygonal domains. }
Thesis (Ph.D.) Massachusetts Institute of Technology. 1988.

\bibitem[Fr]{Fr} D. Fried,
\emph{Cyclic resultants of reciprocal polynomials}, Holomorphic dynamics
(Mexico, 1986), 124--128, Lecture Notes in Math., \textbf{1345}, Springer,
Berlin, 1988.

\bibitem[Gh04]{Gh04} Ghomi, M. Shortest periodic billiard trajectories in convex bodies. Geom. Funct. Anal. 14 (2004), no. 2, 295-302. 

\bibitem[GG73]{GG73} M. Golubitsky and V.  Guillemin, {\it Stable mappings and their singularities}. Graduate Texts in Mathematics, Vol. 14. Springer-Verlag, New York-Heidelberg, 1973.

\bibitem[GuMe79]{GuMe79} V. Guillemin and R. B. Melrose, \emph{The Poisson summation formula for manifolds with boundary}. Adv. in Math. \textbf{32} (1979), no. 3, 204 - 232.

\bibitem[GuMe79b]{GuMe79b}  V. Guillemin and R. B.  Melrose,  A cohomological invariant of discrete dynamical systems. E. B. Christoffel (Aachen/Monschau, 1979), pp. 672?679, Birkh\"auser, Basel-Boston, Mass., 198.

\bibitem[HeLuRo17]{HLR1} Hezari, Hamid; Lu, Zhiqin; Rowlett, Julie, \emph{The Neumann isospectral problem for trapezoids.}
Ann. Henri Poincar\'e \textbf{18} (2017), no. 12, 3759--3792. 

\bibitem[HeLuRo20]{HLR2} Hezari, Hamid; Lu, Zhiqin; Rowlett, Julie, \emph{The Dirichlet isospectral problem for trapezoids.}, arXiv: 2009.00714, 2020. 

\bibitem[HeZe10]{HZ}  H. Hezari and S. Zelditch,  \emph{Inverse spectral problems for $(\mathbb Z/2 \mathbb Z)^n$-symmetric domains in $\mathbb R^n$},  GAFA.  \textbf{20}(2010), no.1, 160--191. 

\bibitem[HaZe12]{HZ12} H. Hezari and S. Zelditch, \emph{$C^\infty$ spectral rigidity of the ellipse}, Anal. PDE \textbf{5} (2012), no. 5, 1105--1132. 

\bibitem[HeZe19]{HZ19} H. Hezari and S. Zelditch, One can hear the shape of ellipses of small eccentricity, arXiv:1907.03882.

\bibitem[HeZe21]{HZ21} H. Hezari and S. Zelditch, A new duality in billiards, (in preparation).

\bibitem[HCMF]{HCMF} Hernández Cifre, María A.; Martínez Fernández, Antonio R. The isodiametric problem and other inequalities in the constant curvature 2-spaces. Rev. R. Acad. Cienc. Exactas Fís. Nat. Ser. A Mat. RACSAM 109 (2015), no. 2, 315?325.



\bibitem[Hir]{Hir}  M. W. Hirsch,  {\it Differential topology}. Corrected reprint of the 1976 original. Graduate Texts in Mathematics, 33. Springer-Verlag, New York, 1994.


\bibitem[KT]{KT} V.V. Kozlov and D. V. Treshchev, {\it Billiards: A Genetic Introduction to the Dynamics
of Systems with Impacts}, Translations of Math. Monographs 89, AMS publications, Providence, R.I. (1991).


\bibitem[K01]{K01} V. V. Kozlov,  Two-link billiard trajectories: extremal properties and stability. (Russian) Prikl. Mat. Mekh. 64 (2000), no. 6, 942-946; translation in J. Appl. Math. Mech. 64 (2000), no. 6, 903-907 (2001).


\bibitem[KrP]{KrP}  Krantz, Steven G.; Parks, Harold R. A primer of real analytic functions. Second edition. Birk\"auser Advanced Texts: Basler Lehrbücher. [Birkh\"auser Advanced Texts: Basel Textbooks] Birkhäuser Boston, Inc., Boston, MA, 2002.

\bibitem[KM]{KM} A. Kriegl, P. W. Michor, A convenient setting for real analytic mappings, Acta Mathematica 165 (1990), 105--159.

 \bibitem[La93]{L} V. F. Lazutkin, \emph{KAM theory and semiclassical approximations to eigenfunctions},  Ergeb. Math. Grenzgeb. (3), \textbf{24}. Springer-Verlag, Berlin, 1993.
 
 \bibitem[ML]{ML} D. B. Massey and D. T.  Le,  Notes on real and complex analytic and semianalytic singularities. {\it Singularities in geometry and topology}, 81-126, World Sci. Publ., Hackensack, NJ, 2007

\bibitem[MaMe82]{MM} Marvizi, Shahla; Melrose, Richard
\emph{Spectral invariants of convex planar regions. }
J. Differential Geom. \textbf{17} (1982), no. 3, 475--502. 

\bibitem[Mo]{Mo}  Mossinghoff, Michael J. Isodiametric problems for polygons. Discrete Comput. Geom. 36 (2006), no. 2, 363-379.


 \bibitem[Pee]{Pee} J. Peetre,
On Hadamard's variational formula. J. Differential Equations 36
(1980), no. 3, 335-346.

\bibitem[PS87]{PS87} V. Petkov and L. Stojanov,  Periods of multiple reflecting geodesics and inverse spectral results. Amer. J. Math. 109 (1987), no. 4, 619-668.



\bibitem[PS17]{PS17} V. M. Petkov and L. N.  Stoyanov,  {\it Geometry of the generalized geodesic flow and inverse spectral problems}. Second edition. John Wiley $\&$ Sons, Ltd., Chichester, 2017

\bibitem[RS16]{RS16} A. Rosenthal and O. Szasz, Eine Extremaleigenschaft der Kurven konstanter Breite,  Jahresber. Deutsch.
Math.-Verein. 25 (1916), 278-282.

\bibitem[S84]{S84} J. Soucek, 
Morse-Sard theorem for closed geodesics.
Comment. Math. Univ. Carolin. 25 (1984), no. 2, 265-272.


\bibitem[SS72]{SS72} J.  Soucek and V. Soucek,  Morse-Sard theorem for real-analytic functions. Comment. Math. Univ. Carolinae 13 (1972), 45-51.

\bibitem[Vig19]{Vig} A. Vig, The wave trace and the Birkhoff billiards, arXiv: 1910.06441, 2019. 

\bibitem[Wa00]{W00} K. Watanabe, Plane domains which are spectrally determined. Ann. Global Anal. Geom. 18 (2000), no. 5, 447--475.

\bibitem[Wa02]{W02} K. Watanabe,  Plane domains which are spectrally determined. II. J. Inequal. Appl. 7 (2002), no. 1, 25--47.

\bibitem[Wi04]{Wi04} J. Winkelmann, Jör
Realizing connected Lie groups as automorphism groups of complex manifolds.
Comment. Math. Helv. 79 (2004), no. 2, 285-299.

\bibitem[Ze09]{Z09} S. Zelditch, 
Inverse spectral problem for analytic domains. II. $\Z_2$-symmetric domains. 
Ann. of Math. (2) 170 (2009), no. 1, 205-269.



\end{thebibliography}
\end{document}